\documentclass[dvipsnames]{article}
\usepackage{amsmath,amssymb}
\usepackage{amsthm}
\usepackage[all]{xy}
\usepackage{tikz}
\usepackage{tikz-cd}
\usetikzlibrary{positioning}

\usetikzlibrary{decorations.markings}

\usepackage{hyperref}

\usepackage{todonotes}

\theoremstyle{definition}
\newtheorem{thm}{Theorem}[section]
\newtheorem{defn}[thm]{Definition}
\newtheorem{ex}[thm]{Example}
\newtheorem{rem}[thm]{Remark}
\newtheorem{prop}[thm]{Proposition}

\newtheorem{cor}[thm]{Corollary}
\newtheorem{lem}[thm]{Lemma}

\newtheorem{open}[thm]{Open Problem}

\usepackage{graphicx}

\colorlet{green}{OliveGreen}

\tikzset{
    graph/.style={thick, red}
}

\newcommand{\SP}{\mathit{SP}}
\newcommand{\Z}{\mathbb{Z}}
\newcommand{\C}{\mathbb{C}}
\newcommand{\N}{\mathbb{N}}
\newcommand{\R}{\mathbb{R}}
\newcommand{\Q}{\mathbb{Q}}

\DeclareMathOperator{\id}{id}
\DeclareMathOperator{\im}{im}

\begin{document}

\title{
At most $n$-valued  maps }


\author{DACIBERG~LIMA~GON\c{C}ALVES\\
Departamento de Matem\'atica - IME-USP,\\
Rua do Mat\~ao 1010, Butant\~a\\
CEP:~05508-090 - S\~ao Paulo - SP - Brazil.\\
e-mail:~\url{dlgoncal@ime.usp.br}\vspace*{4mm}\\
 ROBERT SKIBA \\
 Faculty of Mathematics and Computer Science\\
Nicolaus Copernicus University in Toru\'n\\
Chopina 12/18\\
87-100 Toru\'n\\
Poland\\
E-mail: robert.skiba@mat.umk.pl\vspace*{4mm}\\
P. CHRISTOPHER  STAECKER\\
Department of Mathematics, Fairfield University,\\
Fairfield, 06823-5195, CT, USA.\\
e-mail:~\url{
cstaecker@fairfield.edu}
}

\date{\today}

\begingroup
\renewcommand{\thefootnote}{}
\footnotetext{Classification (MSC2020):  (Primary) 54C60; (Secondary) 55R80, 57M05, 55P42. }  
\endgroup

\maketitle

\begin{abstract}
This paper concerns various models of ``at-most-$n$-valued maps''. That is, multivalued maps $f:X\multimap Y$ for which $f(x)$ has cardinality at most $n$ for each $x$. We consider 4 classes of such maps which have appeared in the literature: $\mathcal U$, the set of exactly $n$-valued maps, or unions of such; $\mathcal F$, the set of $n$-fold maps defined by Crabb; $\mathcal S$, the set of symmetric product maps; and $\mathcal W$, the set of weighted maps with weights in $\mathbb N$. Our main result is roughly that these classes satisfy the following containments:
\[ \mathcal U \subsetneq \mathcal F \subsetneq \mathcal S = \mathcal W \]

Furthermore we define the general class $\mathcal C$ of all at-most-$n$-valued maps, and show that there are maps in $\mathcal C$ which are outside of any of the other classes above. We also describe a configuration-space point of view for the class $\mathcal C$, defining a configuration space $C_n(Y)$ such that any at-most-$n$-valued map $f:X\multimap Y$ corresponds naturally to a single-valued map $f:X\to C_n(Y)$. We give a full calculation of the fundamental group and homology groups of $C_n(S^1)$.
\end{abstract}

\section{Introduction}
Multivalued maps are an important and broad subject. Many problems and applications have motivated the study of certain types of multivalued maps. Notably, for a fixed positive integer $n$ we have the so called $n$-valued maps
$X \multimap Y$, which consist of the multivalued maps where the image of a point of the domain is a set of cardinality exactly $n$. We also refer to these as \emph{equicardinal} maps. Such maps have been intensively studied and typical topological techniques have been successfully used, including a configuration space viewpoint which views an $n$-valued multimap $X\multimap Y$ as a single-valued map from $X$ into the unordered configuration space of $n$ points on $Y$.

This paper concerns the slightly more general class of so called at-most-$n$-valued maps
$X \multimap Y$, the multivalued maps where the image of a point of the domain is a set of cardinality at most $n$. To exemplify such maps consider $p:\hat X \to X$
a branched covering with $n$-sheets.  The map $x \to p^{-1}(x)$ is at-most-$n$-valued, and is $n$-valued if and only if $p$ is a covering. Another important example is the multimap from the space of all complex polynomials of degree $n$ into the complex numbers $\mathbb{C}$ which associates to each polynomial the set of its roots. This is an at-most-$n$-valued map.

In order to study some natural questions, many mirrored by results of $n$-valued maps, we initially study a suitable configuration space where single valued maps into this space will correspond to at-most-$n$-valued maps. In existing literature there are other closely related classes of at-most-$n$-valued maps. We consider the following four families of maps:
\begin{itemize}
\item $\mathcal{U}$, the set of unions of equicardinal maps
\item $\mathcal{F}$, the set of $n$-fold maps as defined by Crabb
\item $\mathcal{W}$, the set of $\N$-weighted maps
\item $\mathcal{S}$, the set of symmetric product maps.
\end{itemize}

It is natural and useful to explore the relations among these families, including the general family of all at-most-$n$-valued maps. Questions concerning one family may potentially be illuminated by using known results from another family. The families above have been already considered for some time, except the family $\mathcal{F}$
which has appeared more recently.

This paper studies several aspects of such multivalued maps. First we define a suitable configuration space for the at-most-$n$-valued maps. Then we give general relations between the various families listed above. We also prove a general statement about a large class of weighted maps which can be realized as sums of single-valued maps. Finally, we calculate the homology and homotopy groups for the at-most-$n$-valued configuration space of the circle $S^1$.

Multivalued maps play a significant role in both general topology and algebraic topology. Surprisingly, they can also be effectively applied to studying the existence of periodic trajectories in ordinary differential equations (examples of such applications can be found in \cite{Conti,Pejs3,skiba}).

\section{A configuration space for at-most-$n$-valued \linebreak maps}
Given sets $X$ and $Y$ and a positive integer $n$, an $n$-valued function from $X$ to $Y$ is a set-valued function $f:X\multimap Y$ such that $f(x)\subseteq Y$ has cardinality exactly $n$ for every $x\in X$. 

We will often wish to avoid naming a specific natural number $n$ when discussing these maps, and use the term \emph{equicardinal}.
\begin{defn}
A set-valued map $f:X\multimap Y$ is \emph{equicardinal} if there is some natural number $n>0$ such that $f$ is $n$-valued.
\end{defn}

For equicardinal maps there is a well studied configuration space point of view from \cite{BrGo}. Let $D_n(Y)$ be the \emph{unordered configuration space of $n$ points in $Y$}, defined as:
\[ D_n(Y) = \{ \{y_1,\dots,y_n\} \mid y_i \in Y, y_i \neq y_j \text{ for } i\neq j \}. \]

When $Y$ is a topological space, we give $D_n(Y)$ a topology as follows: begin with the product topology on $Y^n$, then consider the subspace $F_n(Y)$ of tuples $(y_1,\dots,y_n)\in Y^n$ with $y_i \neq y_j$ for $i\neq j$. This is the \emph{ordered configuration space}. Then $D_n(Y)$ is the quotient of $F_n(Y)$ up to ordering, and so its topology is given by the quotient topology. It is shown in \cite{BrGo} that every upper and lower semicontinuous map $f:X\multimap Y$ naturally corresponds to a continuous single-valued map $F:X\to D_n(Y)$.

Recall that a multivalued function $f : X \multimap Y$ is \emph{lower semi-continuous} if $V$ an open subset of $Y$ implies
that the set $\{x \in  X\mid f(x)\cap V\ne \emptyset\}$ is open in $X$. This function $f$ is \emph{upper semi-continuous} if $V$ open in $Y$ implies that the set $\{x \in X \mid f(x) \subset V\}$ is open in $X$.

Our main focus for this paper is the following more general class of maps:
\begin{defn}
Given spaces $X, Y$ and some natural number $n>0$, we say $f:X\multimap Y$ is an \emph{at-most-$n$-valued map} when it is both lower and upper semicontinuous and $f(x)\subset Y$ has nonzero cardinality at most $n$ for each $x\in X$.
\end{defn}

Proposition 2.1 of \cite{BrGo} shows that, for an $n$-valued function with Hausdorff codomain, lower semicontinuity implies upper semicontinuity. We obtain the same result for at-most-$n$-valued functions:
\begin{thm}\label{lowerisupper}
Let $Y$ be Hausdorff, and let $f:X\multimap Y$ be at-most-$n$-valued and lower semicontinuous. Then $f$ is also upper semicontinuous.
\end{thm}
\begin{proof}
We use essentially the same proof as in \cite{BrGo}. Let $V\subset Y$ be an open set, and we will show that the set $U = \{x\in X \mid f(x)\subset V\}$ is open in $X$. Take some $x\in U$, and we will construct a neighborhood of $x$ contained in $U$.

Let $f(x) = \{y_1,\dots, y_k\}$ for $k\le n$. Since $Y$ is Hausdorff, there are disjoint open sets $V_i \subset V$ with $y_i\in V_i$ for each $i$. Since $V_i$ is open and $f$ is lower semicontinuous, the following sets $W_i$ are open for each $i\in \{1,\dots,k\}$:
\[ W_i = \{a \in X \mid f(a) \cap V_i \neq \emptyset \}, \]
and $x \in W_i$ for each $i$.

Let $W = \bigcap W_i$, so that $W$ is an open set containing $x$, and $W \subseteq U$. Thus $W$ is the desired open neighborhood of $x$ contained in $U$.
\end{proof}

The converse to the theorem above is not generally true for at-most-$n$-valued maps, or even for equicardinal maps.
Let $f:\R \multimap \R$ be the 2-valued function given by:
\[ f(x) = \begin{cases} \{0,-x\} &$ if $ x < 0, \\
\{0,1\} &$ if $x\ge 0. \end{cases} \]
See Figure \ref{uscnotlscfig}. This is a 2-valued function which is upper semicontinuous but not lower semicontinuous.
\begin{figure}
\[
\begin{tikzpicture}[scale=1]
\draw (-2,0) -- (2,0);
\draw (0,-.1) -- (0,2);

\draw[red,thick] (-2,0) -- (2,0);
\draw[red,thick] (-2,2) -- (0,0);
\draw[red,thick] (0,1) -- (2,1);
\filldraw[red] (0,1) circle (.05cm);
\end{tikzpicture}
\]
\caption{The graph of a 2-valued function $f:\R \multimap \R$ which is upper semicontinuous but not lower semicontinuous.\label{uscnotlscfig}}
\end{figure}
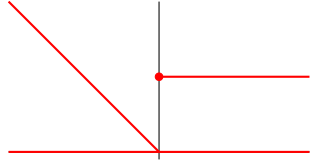

We will see later in Lemma \ref{conti-N} that a converse to Theorem \ref{lowerisupper} does hold for weighted maps: that is, if a weighted map is assumed only to be upper semicontinuous, then it is automatically lower semicontinuous.

In order to construct a configuration-space viewpoint for at-most-$n$-valued maps, we make the following definition:
\begin{defn}
For a space $X$ and some natural number $n>0$, let $C_n(X)$ be the set of all nonempty subsets of $X$ having cardinality at most $n$, topologized as the quotient of $X^n$ under the map $p:X^n \to C_n(X)$ which takes a tuple $(x_1,\dots,x_n)$ to the set $\{x_1,\dots,x_n\}$. We call $C_n(X)$ the \emph{[unordered] configuration space of at most $n$ points on $X$}.
\end{defn}

This space $C_n(X)$ is defined as a quotient space of $X^n$. If we further assume that $X$ is a metric space, since $C_n(X)$ is a space of subsets, we may also consider the topology on $C_n(X)$ induced by the Hausdorff metric.

We review some details about the Hausdorff metric. Given a metric space $(X,d)$ and a point $x\in X$ and a finite subset $A\subset X$, let $d(x,A) = \min_{a\in A} d(x,a)$. Given two finite subsets $A$ and $B$, the Hausdorff distance from $A$ to $B$, denoted $d_H(A,B)$, is the maximum distance from any point of $A$ to its nearest point in $B$, or from any point of $B$ to its nearest point in $A$. Specifically:
\[ d_H(A,B) = \max\left\{\max_{a\in A} d(a,B), \max_{b\in B} d(A,b)\right\}. \]
This function $d_H$ is a metric on the space of all finite subsets of $X$, called the Hausdorff metric.

We will see that the topology on $C_n(X)$ as a quotient of $X^n$ is the same as the topology on $C_n(X)$ induced by the Hausdorff metric. Towards that end, we prove a lemma about open balls in the Hausdorff metric on $C_n(X)$.
\begin{lem}\label{hausdorffball}
Let $X$ be a metric space, and let $d_H$ denote the Hausdorff metric on $C_n(X)$.
Let $x = p(x_1,\dots,x_n) \in C_n(X)$, and let $\epsilon > 0$. Then:
\[ B_H(x,\epsilon) = p(B(x_1,\epsilon) \times \dots \times B(x_n,\epsilon)) \cap \bigcap_i p(B(x_i,\epsilon) \times X \times \dots \times X), \]
where $B_H(x,\epsilon)$ denotes the open $d_H$-ball of radius $\epsilon$ about $x$ in $C_n(X)$, and $B(x_i,\epsilon)$ is the open ball of radius $\epsilon$ in $X$.
\end{lem}
\begin{proof}
We begin by showing that $B_H(x,\epsilon) = A$, where $A$ is the following set:
\[ A = \left\{ z\in C_n(X) \mid z \subset \bigcup_{i} B(x_i,\epsilon)  \text{ and } z \cap B(x_i,\epsilon) \neq \emptyset \text{ for each $i$}  \right\}. \]

Take $y = p(y_1,\dots,y_n) \in B_H(x,\epsilon)$, so that $d_H(x,y) < \epsilon$. By the definition of $d_H$, this means that both $\max_{x_i\in x} d(x_i,y)$ and $\max_{y_i\in y} d(x,y_i)$ are less than $\epsilon$. The first statement, that $\max_{x_i\in x} d(x_i,y) <\epsilon$, means that each point of $x$ is within $\epsilon$ of some point of $y$. That is, $y \cap B(x_i,\epsilon) \neq \emptyset$ for each $i$. The second statement, that $\max_{y_i\in y} d(x,y_i)< \epsilon$, means that each point of $y$ is within $\epsilon$ of some point of $x$. That is, $y \subset \bigcup_{i} B(x_i,\epsilon)$. Thus $y\in B_Y(x,\epsilon)$ is equivalent to $y \in A$ as desired.

Writing $A$ as an intersection of two sets, so far we have shown that:
\[
B_H(x,\epsilon) = \left\{ z \in C_n(X) \mid z \subset \bigcup_i B(x_i,\epsilon) \right\} \cap \{ z \mid z \cap B(x_i,\epsilon) \neq \emptyset \text{ for each $i$}\}.
\]

The first set above equals $p(B(x_1,\epsilon) \times \dots \times B(x_n,\epsilon))$, and the second set can be expressed as:
\[ \bigcap_i \{ z\mid z \cap B(x_i,\epsilon) \neq \emptyset\} = \bigcap_i p(B(x_i,\epsilon) \times X \times \dots \times X), \]
giving the desired formula.
\end{proof}

Now we are ready to show that the quotient topology on $C_n(X)$ agrees with the topology given by the Hausdorff metric.
\begin{thm}
Let $X$ be a metric space. Then the topology on $C_n(X)$ as a quotient of $X^n$ is equivalent to the topology on $C_n(X)$ given by the Hausdorff metric.
\end{thm}
\begin{proof}
First we argue that an open ball in the Hausdorff metric is also open using the quotient topology. This follows immediately from Lemma \ref{hausdorffball}, which expresses an open ball in the Hausdorff metric as a finite intersection of sets which are open in the quotient topology.

Now for the converse, we take a subset $U \subset C_n(X)$ which is open in the quotient topology, and we will show that it is also open in the Hausdorff metric. Take some $x\in U$, and we will find a $d_H$-ball around $x$ contained in $U$. Let $x = p(x_1,\dots,x_n)$, and since $x\in U$, there is a basis neighborhood containing $x$ of the form:
\[ V = p(B(x_1,\epsilon) \times \dots \times B(x_1,\epsilon)) \subseteq U. \]
Then by Lemma \ref{hausdorffball} we have $B_H(x,\epsilon) \subset V \subset U$. Thus $B_H(x,\epsilon)$ is the desired neighborhood of $x$ contained in $U$, and so $U$ is open.
\end{proof}

In the remainder of the section we complete the description of $C_n(X)$ as a configuration space for at-most-$n$-valued maps by showing that every at-most-$n$-valued map $f:X\multimap Y$ corresponds naturally to a single-valued map $F:X\to C_n(Y)$ given by $F(x) = f(x) \subset Y$.

\begin{lem}
Let $Y$ be Hausdorff, and let $f:X\multimap Y$ be upper semicontinuous and at-most-$n$-valued. Then the corresponding function $F:X\to C_n(Y)$ is continuous.
\end{lem}
\begin{proof}
Recall that the topology on $C_n(Y)$ is induced by the quotient map $p:Y^n \to C_n(Y)$ given by viewing each $n$-tuple as an unordered $k$-element set for $k\le n$ by discarding duplicate elements.
Thus a basis open set $U\subset C_n(Y)$ will have the form $U=p(V_1\times \dots \times V_n)$ where each $V_i \subset Y$ is open. For a particular configuration $y\in C_n(Y)$, we will have $y \in U$ if and only if $y\subset \bigcup V_i$.

To show that $F$ is continuous, choose $U = p(V_1\times \dots \times V_n)$ a basis open set of $C_n(Y)$, and we will show that $F^{-1}(U)$ is open in $X$. Take some point $x \in U$, and we will show that there is a neighborhood $Z\subset F^{-1}(U)$ with $x\in Z$.

Let $f(x) = \{y_1,\dots, y_k\}$ for $k\le n$ and $y_i\neq y_j$ for all $i\neq j$. We may renumber the $V_i$ so that $y_i \in V_i$ for each $i \in \{1,\dots,k\}$. Since $Y$ is Hausdorff, there are pairwise disjoint open sets $W_i\subseteq V_i$ such that $y_i \in W_i$ for each $i$. Let $W = \bigcup W_i$, and let $Z = \{ x\in X \mid f(x)\subset W\}$. This set $Z$ contains the point $x$, and $Z$ is open because $f$ is upper semicontinuous. It remains only to show that $Z \subset F^{-1}(U)$.

Taking $z\in Z$, we have $f(z) \subset W \subseteq \bigcup V_i$, which means that $f(z) \in U$. Thus $F(Z) \subset U$, and so $Z \subset F^{-1}(U)$ as desired.
\end{proof}

\begin{lem}
Let $X$ and $Y$ be metric spaces, and let $F:X\to C_n(Y)$ be continuous. Then the corresponding function $f:X\multimap Y$ is upper and lower semicontinuous.
\end{lem}
\begin{proof}
By Lemma \ref{lowerisupper} it suffices to show that $f$ is lower semicontinuous. For this proof we will view $C_n(Y)$ as having the Hausdorff topology. To show that $f$ is lower semicontinuous, we start with some basis open set $B(z,\epsilon)$ for $z\in Y$, and we must show that the following set in $X$ is open:
\[ U = \{ x \in X \mid f(x) \cap B(z,\epsilon) \neq \emptyset \} \]
Take some $x\in U$, and we will find a neighborhood of $x$ contained in $U$.

Since $x\in U$ we have $f(x)\cap B(z,\epsilon) \neq \emptyset$, so let $y\in f(x)$ be a point within distance $\alpha <\epsilon$ of $z$. Since $F$ is continuous with the Hausdorff metric, there is some $\delta>0$ such that:
\[ F(B(x,\delta)) \subseteq B_H(f(x),\epsilon-\alpha). \]
The above means that, for all $a \in B(x,\delta)$, the set $f(a)$ contains some point within distance $\epsilon-\alpha$ of $y$, and thus, contains some point within distance $\epsilon$ of $z$. That is, if $a \in B(x,\delta)$, then $f(a) \cap B(z,\epsilon) \neq \emptyset$.

This means that $B(x,\delta)$ is a neighborhood of $x$ contained in $U$, and thus $U$ is open as desired.
\end{proof}

The two previous lemmas combine to give the following characterization:
\begin{thm}\label{mapcorrespondence}
Let $X$ and $Y$ be metric spaces. Then an at-most-$n$-valued function $f:X\multimap Y$ is lower and upper semicontinuous if and only if the corresponding single-valued function $F:X\to C_n(Y)$ is continuous.
\end{thm}
Because of the above, we will typically view an at-most-$n$-valued map $f:X\multimap Y$ as identified with the corresponding single-valued map into the configuration space $C_n(Y)$.

\section{Models of at-most-$n$-valued maps in the literature}
For the rest of the paper, when considering multimaps $X\multimap Y$, we will assume $X$ and $Y$ are metric spaces, and $X$ is connected. We will not typically use the metrics directly, but we will use Theorem \ref{mapcorrespondence} implicitly throughout.

In this section we describe 4 different models appearing in the literature for various types of at-most-$n$-valued maps:
\begin{itemize}
\item $\mathcal{U}$, the set of unions of equicardinal maps
\item $\mathcal{F}$, the set of $n$-fold maps as defined by Crabb
\item $\mathcal{W}$, the set of $\N$-weighted maps
\item $\mathcal{S}$, the set of symmetric product maps
\end{itemize}

Let $\mathcal C$ be the set of all general at-most-$n$-valued maps: that is, continuous maps $X\to C_n(Y)$ for some spaces $X$ and $Y$.

Each model $\mathcal X$ carries with it an \emph{interpretation function} $\Psi_{\mathcal{X}}:\mathcal X \to \mathcal C$, which describes how to interpret an element of $\mathcal X$ as an at-most-$n$-valued map. When there is no ambiguity, we write $\Psi_{\mathcal X}$ simply as $\Psi$.

We wish to compare these various models with one another.
For example: does every union of equicardinal maps naturally correspond to an $n$-fold map? This question is equivalent to asking if $\Psi(\mathcal U) \subseteq \Psi(\mathcal F)$.
Our results concerning these 4 models can be summarized by the following relations:
\begin{equation}\label{containments}
\Psi(\mathcal{U}) \subsetneq \Psi(\mathcal{F}) \subsetneq \Psi(\mathcal{S}) = \Psi(\mathcal{W}) \subsetneq \mathcal{C}.
\end{equation}

In the following subsections we define the classes above and describe the interpretation function $\Psi:\mathcal X\to \mathcal{C}$ for each.

\subsection{The class $\mathcal U$: unions of equicardinal maps}
We define the class $\mathcal U$ as the collection of all finite sets $f = \{f_1,\dots,f_m\}$, where each $f_i$ is an equicardinal map $f_i:X\to D_{k_i}(Y)$. Our interpretation function views such a set as a \emph{union of equicardinal maps} as follows:
\[ \Psi(\{f_1,\dots, f_m\}) (x) = f_1(x) \cup \dots \cup f_m(x). \]
When $k_1 + \dots + k_m = n$, the  union above defines an at-most-$n$-valued map $X\multimap Y$.

A homotopy of maps in $\mathcal U$ is defined as follows: Two elements of $\mathcal U$ are homotopic if they can be expressed as $\{f_1 \cup \dots \cup f_m\}$ and $\{f'_1 \cup \dots \cup f'_m\}$ such that $f_i$ is homotopic as a $k_i$-valued map to $f'_i$ for each $i$.

\subsection{The class $\mathcal F$: Crabb's $n$-fold maps}
Crabb, in \cite{crabb}, defines a class of maps as follows:
Let $p:\bar X \to X$ be an $n$-fold covering space, and let $\bar f:\bar X \to Y$ be a single-valued map. Then this determines a multivalued map $F:X \multimap Y$ defined by $F(x) = \bar f(p^{-1}(x))$. Because $f$ is not necessarily injective, we have $\#F(x) \le n$ for all $x$. Crabb uses the notation $F = \bar f/p$. He calls this simply an ``$n$-valued map'', or in some later work a ``structured $n$-valued map.'' We would like to suggest the term ``$n$-fold map'', since these maps are constructed using $n$-fold coverings.

The class $\mathcal F$ consists of all such $n$-fold maps $\bar f/p$ where $p:\bar X \to X$ is a finite cover and $\bar f:\bar X \to Y$ is a map. The interpretation function $\Psi_{\mathcal F}: \mathcal F \to \mathcal C$ is given by $\Psi_{\mathcal F}(\bar f/p) = \bar f\circ p^{-1}$.

Crabb's work does not discuss general homotopies of $n$-fold maps. He demonstrates a homotopy-invariance property for the  fixed point index with respect to homotopies of the form $(\bar f/p) \simeq (\bar f'/p)$, in which the covering $p$ must remain the same while varying only the map $\bar f$. This will not suffice for our purposes, since we want to consider two $n$-fold maps to be homotopic when they use different homeomorphic variations of the same covering map.

We define a homotopy of maps in $\mathcal F$ as follows: Let $(\bar f/p)$ and $(\bar f'/p')$ be two $n$-fold maps where $p:\bar X \to X$ and $p': \bar X' \to X$ are $n$-fold covers and $\bar f: \bar X \to Y$ and $\bar f': \bar X'\to Y$ are maps. We say that $(\bar f/p)$ and $(\bar f'/p')$ are homotopic when there is a homeomorphism 
$h:\bar X \to \bar X'$ such that $f \simeq \bar f' \circ h$ and $p = p' \circ h$.

\bibliographystyle{plain}

\subsection{The class $\mathcal{W}$: $\N$-weighted maps}

Between 1958 and 1961, G. Darbo published the first of three papers on the homology theory of weighted maps. In this work, he introduced a class of multivalued transformations known as weighted mappings (cf. \cite{darbo1, darbo2, darbo3}).
Informally speaking, an $\N$-weighted map from $X$ to $Y$ is a continuous multimap $f:X\multimap Y$ such that $f(x)$ is a finite set for each $x$, and each individual element of the graph of $f(x)$ carries a natural number ``weight." The function describing the weight at each point must have constant sum on neighborhoods of the graph. See \cite{Pejs3} for a survey. Formally, we have:
\begin{defn}\label{definitionofweightedmap1309}
A weighted map from $X$ to $Y$ is a pair
$\psi=(n_{\psi},w_{\psi})$ satisfying the following
conditions:
\begin{enumerate}
\item[$(1)$] $n_{\psi}\colon X \multimap Y$ is a multivalued
upper semicontinuous map such that $n_{\psi}(x)$ is a finite
subset of $Y$ for all $x\in X$;
\item[$(2)$] $w_{\psi}\colon X\times Y\rightarrow \N$ is a
function with the following properties:
\begin{enumerate}
\item[$(i)$] $ w_{\psi}(x,y)=0$ for any $y\not\in n_{\psi}(x)$;
\item[$(ii)$] for any open subset $U$ of $Y$ and $ x \in X $ such that $
n_{\psi}(x) \cap \partial U = \emptyset$ there exists an open
neighborhood V of the point $x$ such that:
\begin{equation}\label{weight}
\sum_{ y \in U} w_{\psi}(x,y)=\sum_{ y \in U} w_{\psi}(z,y),
\end{equation}
for every $z\in V$.
\end{enumerate}
\end{enumerate}
If a weighted map $\psi: X\multimap Y$ satisfies additionally the following condition
\[
w_{\psi}(x,y) >  0  \text{ for  all }  y\in n_{\varphi}(x),
\]
then $\psi$  is called an $\N$-weighted map.
\end{defn}

\begin{lem}\label{conti-N}
If $\psi=(n_{\psi},w_{\psi}): X\multimap Y$ is an $\mathbb{N}$-weighted map, then $n_{\psi}: X\multimap Y$ is continuous.
\end{lem}
\begin{proof}
It suffices to prove that $n_{\psi}$ is lower semicontinous. Namely, we need to show that for any open subset $U$ of $Y$ the following
set
\begin{equation*}
(n_{\psi})^{-1}_+(U):=\{x\in X\mid n_{\psi}(x)\cap U\neq\emptyset\}
\end{equation*}
is open in $X$. To see this, take $x_0\in (n_{\psi})^{-1}_+(U)$ and let
\begin{equation*}
n_{\psi}(x_0)\cap U=\{y_1,...,y_s\}\text{ and }n_{\psi}(x_0)\cap (Y\setminus U)=\{z_1,...,z_r\}.
\end{equation*}
Then there exist two open subsets $V_1$ and $V_2$ of $Y$ such that
\begin{align*}
\overline{V_1}\cap \overline{V_2}=\emptyset &\text{ and }\overline{V_1}\subset U,\\
n_{\psi}(x_0)\cap U\subset V_1&\text{ and }n_{\psi}(x_0)\cap (Y\setminus U)\subset V_2
\end{align*}
and
\begin{equation}\label{weight1}
\sum_{y\in V_1}w_{\psi}(x_0,y)>0.
\end{equation}
Since $n_{\psi}(x_0)\cap \partial V_1=\emptyset$, it follows from \eqref{weight} that there exists an open neighborhood $W$ of $x_0$ such that
\begin{equation}\label{weight2}
\sum_{ y \in V_1} w_{\psi}(x_0,y)=\sum_{ y \in V_1} w_{\psi}(x,y),
\end{equation}
for all $x\in W$. Now taking into account \eqref{weight1} and \eqref{weight2}, we obtain that
\begin{equation*}
n_{\psi}(x)\cap V_1\neq\emptyset,
\end{equation*}
for all $x\in W$, which implies that $W\subset (n_{\psi})^{-1}_+(U)$. This completes the proof that $(n_{\psi})^{-1}_+(U)$ is open
in $X$.
\end{proof}
\begin{defn}
Given an $\N$-weighted map $\psi:X\multimap Y$, where $X$ is a connected (Hausdorff) space, the \emph{weighted index} of $\psi$ is the natural number $I_w(\psi)\in \N$ defined as follows: choose some $x\in X$, and let
\[I_w(\psi): = \sum _{y\in Y}\omega(x,y). \]

\end{defn}
We will now present a lemma that shows that the above definition is correct.

\begin{lem}
When $X$ is connected, the number $I_w(\psi)$ is independent of the choice of $x$.
\end{lem}

\begin{proof}
Assume, for the sake of contradiction, that there exist two points \( x_1, x_2 \in X \) such that
\[
\sum_{y \in Y} w(x_1, y) \neq \sum_{y \in Y} w(x_2, y).
\]
Define the following sets
\[
X_1 := \left\{ x \in X \;\middle|\; \sum_{y \in Y} w(x, y) = \sum_{y \in Y} w(x_1, y) \right\},
\]
and
\[
X_2 := \left\{ x \in X \;\middle|\; \sum_{y \in Y} w(x, y)\neq\sum_{y \in Y} w(x_1, y) \right\}.
\]
It is clear that the sets $X_1$ and $X_2$ have the following properties:
\begin{itemize}
\item $X_1\neq\emptyset$ and $X_2\neq\emptyset$,
\item $X_1 \cap X_2 = \emptyset$ and $X_1\cup X_2=X$,
\end{itemize}
which implies that  $X$ is the disjoint union of two non-empty sets. Now we will show that $X_1$ and $X_2$ are open in $X$.
Indeed, let $x\in X_1$. We have to show that there exists an open neighborhood $V_x$ of $x$ such that $V_x\subset X_1$.
Since $\partial Y=\emptyset$, it follows that $n_{\psi}(x)\cap\partial Y=\emptyset$, and hence, by Definition \ref{definitionofweightedmap1309},
there exists an open subset $V_x$ of $X$ such that $x\in V_x$ and
\begin{equation*}
\sum_{ y \in Y} w_{\psi}(x,y)=\sum_{ y \in Y} w_{\psi}(z,y) \text{ for all }z\in V_x,
\end{equation*}
which implies that $V_x\subset X_1$.  Thus $X_1$ is open in $X$. Let $\tilde{x}\in X_2$. Then
\begin{equation*}
\sum_{y \in Y} w(\tilde{x}, y)\neq\sum_{y \in Y} w(x_1, y).
\end{equation*}
By applying once again the second condition of Definition \ref{definitionofweightedmap1309}, we can conclude that there exists an open
neighborhood $V_{\tilde{x}}$ of $\tilde{x}$ in $X$ such that
\begin{equation*}
\sum_{y \in Y} w(x_1, y) \neq\sum_{ y \in Y} w_{\psi}(\tilde{x},y)=\sum_{ y \in Y} w_{\psi}(z,y) \text{ for all }z\in V_{\tilde{x}} ,
\end{equation*}
which implies that $V_{\tilde{x}}\subset X_2$, and therefore $X_2$ is open in $X$.
This contradicts the connectedness of $X$, so our assumption must be false. Thus, we deduce that the sum
\[
\sum_{y \in Y} w(x, y)
\]
is constant.
\end{proof}

%
%

We must show that \( n_\psi(X) \) is indeed an element of \( \mathcal{C} \), meaning that it is continuous as an at-most-\( n \)-valued map.
\begin{thm}
Let $\psi = (n_\psi,w_\psi)$ be an $\N$-weighted map with weighted index $n$ from $X$ to $Y$. Then $\Psi(\psi) = n_\psi$ is continuous as an at-most-$n$-valued map.
\end{thm}
\begin{proof}
By Definition \ref{definitionofweightedmap1309}, the multivalued map $n_\psi$ is assumed to be upper semicontinuous. Lemma \ref{conti-N} shows that it is also lower semicontinuous.
\end{proof}

\subsection{The class $\mathcal S$: Symmetric product maps}
Let $\SP^n(X)$ be the $n$-fold symmetric product of $X$. This is the quotient of $X^n$ by the symmetric group. We will write elements of $\SP^n(X)$ as $[x_1,\dots,x_n]$ for $x_i\in X$. Let $u:\SP^n(X) \to C_n(X)$ be the map which simply considers an element of $\SP^n(X)$ as a set of at most $n$ points, that is, $u([x_1,\dots,x_n]) = \{x_1,\dots, x_n\}$.

More references about the homotopy theory of Symmetric product maps will given in some of the next sections.

The class $\mathcal S$ consists of all maps $f:X\to \SP^n(Y)$. The function $\Psi_{\mathcal S}: \mathcal S \to \mathcal C$ is given by $\Psi_{\mathcal S}(f) = u \circ f$.
\begin{thm}
Let $f: X\multimap \SP^n(Y)$ be a symmetric product map. Then $u\circ f:X\to C_n(Y)$ is continuous as an at-most-$n$-valued map.
\end{thm}
\begin{proof}
It will suffice to show that $u: \SP^n(Y) \to C_n(Y)$ is continuous.

Let $q:Y^n \to \SP^n(Y)$ be the quotient by the symmetric group, and $p:Y^n \to C_n(Y)$ be the quotient defining $C_n(Y)$. Clearly we have $q \circ u = p$. Then if $U \subset C_n(Y)$ is an open set, we have $p^{-1}(U)$ open because $p$ is continuous, and thus $q(p^{-1}(U))$ is open because $q$ is an open map. But we have
\[ q(p^{-1}(U)) = q(q^{-1}(u^{-1}(U))) = u^{-1}(U), \]
so $u^{-1}(U)$ is open and thus $u$ is continuous.
\end{proof}

Finally we show that every at-most-$2$-valued map can be obtained by a symmetric product map. Thus for $n=2$ we can view the classes of symmetric product maps and at-most-$n$-valued maps as the same. This is not true for $n>2$, as we will see in Example \ref{forkmap}, which presents an at-most-3-valued map which cannot be realized by a symmetric product map.

\begin{prop}\label{samen=2}
Let $X$ be a space. Then $C_2(X)$ is homeomorphic to $\SP^2(X)$.
\end{prop}
\begin{proof}
Viewing the symmetric product as a quotient of the Cartesian product, the elements of  $\SP^2(X)$ are classes of elements of $X\times X$ of the form $(x,x)$ or $(x_1,x_2)$ with $x_1\ne x_2$. The class $[(x,x)]$ of an element of the form $(x,x)$ contains only the element $(x,x)$. The class $[(x_1,x_2)]$ of an element $(x_1,x_2)$ with $x_1\ne x_2$ contains exactly two elements, namely $(x_1,x_2)$ and $(x_2,x_1)$.

From  the definition of $C_2(X)$, we have a completely analogous description of its elements: The elements of $C_2(X)$ are classes of elements of the form $\{x\}$ or $\{x_1,x_2\}$ for $x_1\neq x_2$. The class in $X\times X$ corresponding to $\{x\}$ contains only $(x,x)$, and the class of $\{x_1,x_2\}$ contains exactly $(x_1,x_2)$ and $(x_2,x_1)$.

The topologies on $C_2(X)$ and $\SP^2(X)$ induced by the quotient from $X^2$ coincide, therefore the map $h:C_2(X) \to \SP^2(X)$ given by $h(\{x\}) = [x,x]$ and $h(\{x_1,x_2\}) = [x_1,x_2]$ is a homeomorphism.
\end{proof}

\section{$\Psi(\mathcal U) \subsetneq \Psi(\mathcal F)$}

\begin{thm}
Any union of equicardinal maps naturally induces an $n$-fold map. That is, $\Psi(\mathcal U) \subseteq \Psi(\mathcal F)$.
\end{thm}

\begin{proof}
Take some at-most-$n$-valued map $f\in \Psi(\mathcal U)$, we will show that $f\in \Psi(\mathcal F)$. Since $\Psi:\mathcal U \to \mathcal C$ is an inclusion map, we may directly consider $f$ itself as a union of equicardinal maps.
Specifically, say $f = f_1 \cup \dots \cup f_k$, where $f_i:X\multimap Y$ is exactly $n_i$-valued and $n_1 + \dots + n_k = n$.

Each map $f_i$ corresponds to an $n_i$-fold map as follows: Let $\Gamma_i \subset X\times Y$ be the graph of $f_i$, given by
\[ \Gamma_i = \{(x,y)\in X\times Y \mid y\in f_i(x) \}. \]
Let $p_i:\Gamma_i \to X$ be the projection map $p_i(x,y) = x$. Since $f_i$ is exactly $n_i$-valued, it can be shown that $p_i:\Gamma_i \to X$ is an $n_i$-fold covering map. Furthermore, there is a single-valued map $\hat f_i: \Gamma_i \to Y$ given by $\hat f_i(x,y) = y$.

Now let $\hat X$ be the disjoint union
\[ \hat X = \bigsqcup_i \Gamma_i \]
and let $p:\hat X \to X$ be the disjoint union of the maps $p_i$. Then $p:\hat X \to X$ is an $n$-fold covering map. Let $\hat f:\hat X \to Y$ be the disjoint union of the maps $\hat f_i$, and we note that $\hat f \circ p^{-1}(x) = f(x)$ for all $x$. Thus $f = \Psi(\hat f/p)$, and so $f \in \Psi(\mathcal F)$ as desired.
\end{proof}

%

\begin{thm}
$\Psi(\mathcal U) \neq \Psi(\mathcal F)$
\end{thm}
\begin{proof}
We must find some spaces $X,Y$, some natural number $n$, and some at-most-$n$-valued map $f:X\multimap Y$ with $f\in \Psi(\mathcal F)$ but $f\not \in \Psi(\mathcal U)$. We take $n=2$ with $X$ the projective plane $X=\R P^2$ and $Y$ the sphere $Y=S^2$. We view $\R P^2$ as $S^2\subset \R^3$ with antipodal points identified. For some point $(x,y,z)\in S^2$, let $[(x,y,z)]$ be the corresponding point of $\R P^2$.

We define a 2-fold map as follows: let $p:S^2 \to \R P^2$ be the standard 2-fold covering, given by $p(x,y,z) = [(x,y,z)]$. Let $\bar f: S^2 \to S^2$ be defined by $f(x,y,z) = (x,y,|z|)$. Then $(\bar f/p)$ is a 2-fold map of $\R P^2 \to S^2$.

Let $f = \Psi(\bar f/p) = \bar f \circ p^{-1}:\R P^2 \multimap S^2$, which is given in coordinates as:
\[ \bar f \circ p^{-1}([(x,y,z)]) = \bar f(\{(x,y,z),(-x,-y,-z)\}) = \{(x,y,|z|),(-x,-y,|z|)\}. \]
Here we see that the set $f([(x,y,z)])$ has cardinality 1 when
\[(x,y,z) \in \{(0,0,1), (0,0,-1)\},\]
and cardinality 2 otherwise.

Clearly $f\in \Psi(\mathcal F)$, and we complete the proof by showing that $f\not\in \Psi(\mathcal U)$, that is, $f$ is not a union of equicardinal maps. According to the description above we see that $f$ is not exactly 2-valued, and we must also show that it is not a union of two single-valued maps.

Let $A = \{[(x,y,z)] \in \R P^2 \mid z = 0\}$. Then $f$ restricts on $A$ to the 2-valued map $f|_A([(x,y,0)]) = \{(x,y,0),(-x,-y,0)\}$. This restriction is equivalent to the map on $S^1$ given by $e^{i\theta} \mapsto \{e^{i\theta},e^{-i\theta}\}$ (viewing $S^1$ as complex numbers of modulus 1), and this 2-valued map cannot be expressed as a union of single-valued maps. Thus $f|_A$ is not a union of two single-valued maps, and so $f$ itself is not a union of two single-valued maps.
\end{proof}

Note that in the example above, the map $\bar f:S^2 \to S^2$ is not surjective, and thus it is homotopic to a constant. This means that $\bar f / p$ is homotopic as an $n$-fold map to a constant map, and so $\Psi(\bar f/p)$ is homotopic to a map in $\mathcal U$, even though it is not itself in $\mathcal U$. It is natural to ask which $n$-fold maps are homotopic to maps from $\mathcal U$ in this way. We do not have general results on this question, but we present the following as an example, featuring an interesting application of the Borsuk-Ulam Theorem.

\begin{prop}
Let $f/p: \R P^m \to \R P^m$ be a $2$-fold map where $p:S^m \to \R P^m$ is the standard double covering, and $f:S^m \to S^m$ has odd degree. Then $f/p$ is homotopic to an $n$-fold map $\alpha/p$ with $\Psi(\alpha/p) \in \Psi(\mathcal U)$.
\end{prop}
\begin{proof}
We use the following form of the Borsuk-Ulam theorem provided  by \cite[Theorem 5.3]{CoVend}: if $k$ is even, then every map $g:S^m\to S^m$ of degree $k$ satisfies $g(x)=g(-x)$ for some $x$; and if $k$ is odd then there is some map $g$ of degree $k$ with $g(x)\neq g(-x)$ for all $x$.

Applied to our present situation, since $f$ has odd degree, there is some map $\alpha:S^n\to S^n$ of the same degree with $\alpha(x)\neq \alpha(-x)$ for all $x$. This means that $\Psi(\alpha/p)(x) = \alpha\circ p^{-1}(x)$ has cardinality exactly 2 for each $x$, and thus $\Psi(\alpha/p) \in \Psi(\mathcal U)$. Since $f$ and $\alpha$ have the same degree we have $f\simeq \alpha$ and thus $f/p \simeq \alpha/p$ as desired.
\end{proof}

\section{$\Psi(\mathcal F) \subsetneq \Psi(\mathcal S)$}
Recall that, by functoriality of the symmetric product, a single-valued map $f:X\to Y$ naturally induces a map of $\SP^n(X) \to \SP^n(Y)$ which we denote by $\SP^n(f)$. In this way, if $g:X\to \SP^n(Y)$ and $f:Y\to Z$, then we may form the composition $\SP^n(f) \circ g: X\to \SP^n(Z)$.
Let $\iota:D_n(X) \to \SP^n(X)$ be the map given by $\iota(\{x_1,\dots, x_n\}) = [x_1,\dots,x_n]$.

Recall that for $f\in \mathcal S$, we have $\Psi(f) = u\circ f$, where $u:\SP^n(X) \to C_n(X)$ is the map $u([x_1,\dots,x_n]) = \{x_1,\dots,x_n\}$.

First we consider a convenient way to interpret a map from $\mathcal F$ directly as a map in $\mathcal S$. Let $i:\mathcal F \to \mathcal S$ be the following map:
\[ i(\bar f/p) = \SP^n(\bar f) \circ \iota \circ p^{-1}. \]
This map $i$ respects the interpretation maps as follows:
\begin{lem}
Let $i:\mathcal F \to \mathcal S$ be as above. Then $\Psi_{\mathcal F} = \Psi_{\mathcal S} \circ i$.
\end{lem}
\begin{proof}
Take $\bar f/p \in \mathcal F$, and let $x\in X$ with $p^{-1}(x) = \{x_1,\dots, x_n\}$. Following the constructions above, we will have $\SP^n(\bar f) \circ \iota \circ p^{-1}(x) = [\bar f(x_1),\dots, \bar f(x_n)]$, and $\bar f \circ p^{-1}(x) = \{f(x_1),\dots, f(x_n)\}$. That is, $\bar f \circ p^{-1} = u \circ \SP^n(\bar f) \circ \iota \circ p^{-1}$. Thus we have:
\[ \Psi_{\mathcal S}(i(\bar f/p)) = u \circ \SP^n(\bar f) \circ \iota \circ p^{-1} = \bar f \circ p^{-1} = \Psi_{\mathcal F}(\bar f/p) \]
as desired.
\end{proof}

It will follow that $\Psi(\mathcal F) \subseteq \Psi(\mathcal S)$:
\begin{thm}
Any $n$-fold map naturally induces a map into the $n$th symmetric product. That is, $\Psi(\mathcal F) \subseteq \Psi(\mathcal S)$.
\end{thm}
\begin{proof}
By the lemma above, we have $\Psi_{\mathcal F}(\mathcal F) = \Psi_{\mathcal S}(i(\mathcal F)) \subset \Psi_{\mathcal S}(\mathcal S)$ as desired.
\end{proof}

Now we proceed to show that $\Psi(\mathcal F) \neq \Psi(\mathcal S)$. This will require a lemma about homotopies of maps in $\mathcal F$ and $\mathcal S$.

\begin{lem}\label{Fhtp}
Let $i:\mathcal F \to \mathcal S$ be given by:
\[ i(\bar f/p) = \SP^n(\bar f) \circ \iota \circ p^{-1}. \]
If $\bar f,\bar f': \bar X \to Y$ are homotopic with some finite cover $p:\bar X \to X$, then $i(f/p)$ and $i(f'/p)$ are homotopic as symmetric product maps.
\end{lem}
\begin{proof}
Let $H:\bar X \times [0,1] \to Y$ be a homotopy from $\bar f$ to $\bar f'$. For each $t\in [0,1]$, let $\bar f_t(x) = H(x,t)$, so that each $(\bar f_t/p)$ is an $n$-fold map with $\bar f_0 = \bar f$ and $\bar f_1 = \bar f'$. Then we can construct a homotopy $H': X\times [0,1] \to \SP^n(Y)$ given by
\[ H'(x,t) = \SP^n(\bar f_t) \circ \iota \circ p^{-1}(x). \]
This $H'$ is a continuous homotopy $H':X\times [0,1] \to \SP^n(Y)$ from $i(\bar f/p)$ to $i(\bar f'/p)$.
\end{proof}

\begin{thm}\label{FneqS}
$\Psi(\mathcal F) \neq \Psi(\mathcal S)$.
\end{thm}
\begin{proof}
Let $i:\mathcal F \to \mathcal S$ be the map defined above.
We must find some spaces $X,Y$, some natural number $n$, and some map $X\to \SP^n(Y)$ which does not correspond to an $n$-fold map. We take $n=2$ with $X = S^5$ and $Y=S^2$.

For spaces $X$ and $Y$, let $[X,Y]$ be the set of all free homotopy classes of (single-valued) maps from $X$ to $Y$, and we will make use of various facts from the literature about symmetric products of spheres, and the homotopy theory of maps of spheres.

Fix a particular 2-fold covering $p:\bar S^5 \to S^5$. Since $S^5$ is simply connected, this cover $\bar S^5$ must be homeomorphic to a disjoint union $\bar S^5 = S^5 \sqcup S^5$. We have:
\[ [S^5 \sqcup S^5, S^2] = [S^5, S^2] \times [S^5, S^2] = \pi_5(S^2) \times \pi_5(S^2), \]
and this homotopy group $\pi_5(S^2)$ is known to be isomorphic to $\Z_2$. Thus there are exactly 4 homotopy classes of 2-fold maps using this cover $p$. That is, there are 4 maps $f_1,f_2,f_3,f_4:\bar S^5 \to S^2$ such that any other map $f:\bar S^5 \to S^2$ must be homotopic to some $f_k$ for $k\in \{1,2,3,4\}$.

Applying $i$, we obtain 4 symmetric product maps $g_1 = i(f_1/p), \dots, g_4 = i(f_4/p)$. Then given any 2-fold map $f/p$, we have:
\begin{equation}\label{f/peq}
i(f/p) \simeq g_k \text{ for some } k\in \{1,2,3,4\}
\end{equation}
where the homotopy above is a homotopy of symmetric product maps.

Since $S^5$ is simply connected, the covering map $p$ is essentially the only 2-fold covering space. Specifically, if $\tilde p: \tilde S^5 \to S^5$ is another 2-fold covering map of $S^5$, then there is a homeomorphism $h: \bar S^5 \to \tilde S^5$ with $\tilde p \circ h = p$. If $\tilde f: \tilde S^5 \to S^2$ is a map, then we will have a map $\tilde f\circ h: \bar S^5 \to S^2$ and:
\[
\begin{split}
i((\tilde f \circ h) / p) &= \SP^n(\tilde f \circ h) \circ \iota \circ p^{-1} = \SP^n(\tilde f) \circ \SP^n(h) \circ \iota \circ p^{-1} \\
&= \SP^n(\tilde f) \circ \iota \circ \tilde p^{-1} = i(\tilde f / \tilde p)
\end{split}
\]
Thus for any 2-fold map $\tilde f/\tilde p$, the above and \eqref{f/peq} gives $i(\tilde f/\tilde p) = i(\tilde f \circ h/p) \simeq g_k$ for some $k\in \{1,2,3,4\}$.

So far we have shown that any symmetric product map $S^5 \to \SP^2(S^2)$ in the image of $i$ is homotopic to one of $g_1,g_2,g_3,g_4$. Next we show that that there are infinitely many homotopy classes of maps $S^5 \to \SP^2(S^2)$.

The paper \cite{blagrz2004} collects some examples of symmetric products of familiar spaces. Example (4) in Section 2.1 of that paper shows that $\SP^n(S^2)$ is homeomorphic to the complex projective plane $\C P^n$. Thus we have $[S^5, \SP^2(S^2)] = [S^5, \C P^2] = \pi_5(\C P^2)$, and this homotopy group is known to be isomorphic to $\Z$. Thus among symmetric product maps $S^5 \to \SP^2(S^2)$, there are infinitely many different homotopy classes.

Thus there is a map $g: S^5 \to \SP^2(S^2)$ which is not homotopic to any of $g_1,\dots g_4$. This means there is a map in $\mathcal S$ which is different from all maps in $i(\mathcal F)$.

By Proposition \ref{samen=2} we know that $\SP^2(X)$ is homeomorphic to $\mathcal C_2(X)$, and thus that $\Psi_{\mathcal S}$ is a bijection. Thus since $g \in \mathcal S - i(\mathcal F)$, we have $\Psi(g) \in \Psi_{\mathcal S}(\mathcal S) - \Psi_{\mathcal S}(i(\mathcal F)) = \Psi_{\mathcal S}(\mathcal S) - \Psi_{\mathcal F}(\mathcal F)$ as desired.
\end{proof}


\section{$\Psi(\mathcal S) = \Psi(\mathcal W)$}
In this section we show that the class of symmetric product maps is the same as the class of $\N$-weighted maps.

For a connected Hausdorff metric space $X$ and some $n\ge2$, the $n$th symmetric group $\Sigma_n$, acts on $X^n$ by the formula
\begin{equation*}
\sigma\cdot (x_1,...,x_n)=(x_{\sigma(1)},...,x_{\sigma(n)}).
\end{equation*}
If $(X^n,d_n)$ is a product metric space, then we define a metric $d_s$ in $\SP^n(X)$ by putting
\begin{equation}\label{metric}
d_s([x_1,...,x_n],[y_1,...,y_n])=\min \{d_n((x_1,...,x_n),\sigma\cdot (y_1,...,y_n))\mid \sigma\in S_n\}.
\end{equation}

\begin{lem}\label{q}
Let $p: X^n\to SP^nX$ be defined by $p((x_1,...,x_n))=[x_1,...,x_n]$. Then $p$ is an open map.
\end{lem}

\begin{proof}
It is straightforward.
\end{proof}

For convenience, we use the following notation for elements of the symmetric product $\SP^n(Y)$: an element of $\SP^n(Y)$ with repeated coordinates will be written as follows:
\[ [y_1^{k_1}, \dots, y_m^{k_m} ] = [\underbrace{y_1,\dots,y_1}_{\text{$k_1$ times}},\underbrace{y_2\dots,y_2}_{\text{$k_2$ times}},\dots,\underbrace{y_{m},\dots,y_{m}}_{\text{$k_{m}$ times}}] . \]

Our goal for this section is to show that $\Psi(\mathcal W) = \Psi(\mathcal S)$. This we do by proving two subset relations.

\begin{thm}\label{WsubsetS}
$\Psi(\mathcal W) \subseteq \Psi(\mathcal S)$
\end{thm}
\begin{proof}
Assume that $\varphi=(n_{\varphi}, w_{\varphi}): X\multimap Y$ is an $\mathbb{N}$-weighted map with weighted index $n$. We must construct a symmetric product map $f:X\to \SP^n(Y)$ satisfying $\Psi(f) = \Psi(\varphi)$. We construct $f$ as follows: Given $x\in X$, let $n_{\varphi(x)} = \{y_1,\dots,y_{m_x}\}$ with $w_\varphi(x,y_i) = k_i >0$ and we will have $n = k_1 + \dots + k_{m_x}$. Then we define:
\begin{equation*}
f(x)=[y_1^{k_1}, y_2^{k_2}, \dots, y_{m_x}^{k_{m_x}}].
\end{equation*}

Recalling the definitions of the interpretation functions $\Psi_{\mathcal W}$ and $\Psi_{\mathcal S}$, it is easy to check that:
\[ \Psi_{\mathcal W}(\varphi)(x) = \{ y \mid w_\varphi(x,y) \neq 0 \} = \{y_1,\dots y_{m_x}\} \]
and
\[ \Psi_{\mathcal S}(f)(x) = u \circ f(x) = \{y_1,\dots,y_{m_x}\}. \]
Thus we will have $\Psi(\varphi) = \Psi(f)$ as desired.

All that remains to show is that $f:X \to \SP^n(Y)$ is continuous as a symmetric product map, and thus that $f$ is indeed an element of the class $\mathcal S$.
Recall from the definition of $\mathbb N$-weighted map that
\begin{equation}\label{nonzero}
w_{\varphi}(x,y)>0\text{ for all }y\in n_{\varphi}(x) \text{ and }x\in X.
\end{equation}
Fix $x_0\in X$, and by the definition of $i$ we will have:
\begin{equation*}
f(x)=[y_1^{k_1}, y_2^{k_2}, \dots, y_{m_{x_0}}^{k_{m_{x_0}}}],
\end{equation*}
where $n_{\varphi}(x_0)=\{y_1,\dots,y_{m_{x_0}}\}$.

Let $V\subset \SP^n (Y)$ be an open neighborhood of $f(x_0)$, and let $B(x,\delta)$ be the open ball of radius $\delta$ around some point $x\in X$. We must show that there exists $\delta>0$ such that
\begin{equation}\label{conti}
f(x)\in V \text{ for all } x\in B(x_0,\delta).
\end{equation}
 Since $V\ni f(x_0)$ is open in $\SP^n (Y)$ and $Y$ is Hausdorff, by \eqref{metric} there exists $\eta>0$ such that
\[ f(x_0)\in p((B(y_1,\eta)^{k_1} \times \dots \times B(y_{m_{x_0}},\eta)^{k_{m_{x_0}}}) \subseteq V \]
and
\[ D(y_i,\eta)\cap D(y_j,\eta)=\emptyset  \text{ for all }i\neq j, \]
where $p: Y^n\to \SP^n(Y)$ is as in Lemma \ref{q}, and $D(y_i,\eta)$ denotes the closed ball.

Observe that the condition \eqref{weight} implies that there exists $\delta>0$ such that
\begin{equation}\label{local}
\sum_{y\in B(y_i,\eta)}w_{\varphi}(x,y)=\sum_{y\in B(y_i,\eta)}w_{\varphi}(x_0,y)=k_i
\end{equation}
for all $x\in B(x_0,\delta)$ and for all $1\leq i\leq k_{m_{x_0}}$.

Letting $U = \bigcup_{i=1}^{m_{x_0}} B(y_i,\eta)$, the above gives:
\[ \sum_{y\in U} w_\phi(x,y) = \sum_{i=1}^{m_{x_0}} k_i = n \]
for all $x\in B(x_0,\delta)$, and by \eqref{nonzero} this means that $n_\varphi(x)\subseteq U$ for all $x\in B(x_0,\delta)$, and
\[ \# n_\varphi(x) \cap B(y_i,\eta) = k_i \]
for all $x\in B(x_0,\delta)$.

The above means that
\[ f(x)\in p((B(y_1,\eta)^{k_1} \times \dots \times B(y_{m_{x_0}},\eta)^{k_{m_{x_0}}}) \subset V \]
for all $x\in B(x_0,\delta)$, which proves \eqref{conti}.
\end{proof}

\begin{thm}\label{SsubsetW}
$\Psi(\mathcal S) \subseteq \Psi(\mathcal W)$.
\end{thm}
\begin{proof}
Let $f:X\to \SP^n(Y)$ be an element of $\mathcal S$. We must construct a $\N$-weighted map $\varphi = (n_\varphi,w_\varphi): X\multimap Y$ with $\Psi(f) = \Psi(\varphi)$. The weighted map $\varphi$ is constructed as follows: for some $x\in X$, let
\[ f(x) = [y_1^{k_1}, y_2^{k_2},\dots, y_{m_x}^{k_{m_x}}], \]
and define:
\[ n_\phi(x) = u \circ f (x) = \{ y_1,\dots, y_{m_x}\}, \]
and
\[ w_\varphi(x,y) = \begin{cases} k_i & $ if $ y = y_i\in n_\varphi(x) \\
0 & $ if $ y\not\in n_\varphi(x) \end{cases} \]

It is routine to check that $\Psi(f) = \Psi(\varphi)$. It remains only to show that $\varphi$ as defined above is continuous as an $\mathbb N$-weighted map.

Using the same arguments as in the proof of Theorem \ref{WsubsetS} one can prove that $n_{\varphi}$ is upper semi-continuous and that $w_{\varphi}$ satisfies all the conditions
from Definition \ref{definitionofweightedmap1309}. The crucial observation is the following: let $[y_{1}^{k_{1}},\dots, y_{s}^{k_{s}}]\in \SP^n(Y)$. Then, using \eqref{metric},
for any open neighborhood $V$ of  $[y_{1}^{k_{1}}, \dots, y_{s}^{k_{s}}]$ there exists $\eta>0$ such that
\[ p(B(y_1,\eta)^{k_1}\times \dots\times B(y_s,\eta)^{k_s}) \subseteq V \]
and
\[ D(y_i,\eta) \cap D(y_j,\eta) = \emptyset \text{ for all } i\neq j, \]
where $p$ is an in Lemma \ref{q} and $D(y,r)$ denotes the closed ball.
\end{proof}

Theorems \ref{WsubsetS} and \ref{SsubsetW} demonstrate that $\Psi(W)=\Psi(S)$.

The results obtained so far already demonstrate that $\Psi(U) \subsetneq \Psi(W)$, though the construction of a specific map in $\Psi(W)$ but not in $\Psi(U)$ is not very concrete. According to Theorem \ref{FneqS} there are certain 2-valued maps $S^2\multimap S^5$ which can be expressed as $\N$-weighted maps, but not as 2-fold maps (and thus not as at-most-2-valued unions of equicardinal maps).

The following is a much easier example of an at-most-2-valued map which is an $\N$-weighted map but not a union of equicardinal maps.
\begin{ex}\label{WnotUexample}
Let $X$ be the quotient of the real interval $[-1,2]$ with identifications $-1\sim 0$ and $2\sim 1$. Thus $X$ resembles two circles joined by an interval. Let $\varphi:X\multimap X$ be the map pictured in Figure \ref{WnotUfig}.

\begin{figure}
\[ \begin{tikzpicture}[scale=2]
\draw (-1,0) -- (2,0);
\draw (0,-1) -- (0,2);
\draw[densely dotted] (-1,-1) grid (2,2);
\draw[graph] (-1,-1) -- (0,-.5) -- (.25,0) -- (.75,1) -- (1,1) -- (2,1.5);
\draw[graph] (-1,-.5) -- (0,0) -- (.25,0) -- (.75,1) -- (1,1.5) -- (2,2);
\draw[graph,blue] (.25,0) -- (.75,1);
\fill[blue] (.25,0) circle (.03cm);
\fill[blue] (.75,1) circle (.03cm);
\node[above left] at (-.5,-.25) {1};
\node[below right] at (-.5,-.75) {1};
\node[below right] at (.5,.5) {2};
\node[above left] at (1.5,1.75) {1};
\node[below right] at (1.5,1.25) {1};
\end{tikzpicture}
 \]
\caption{An $\N$-weighted map which is not a union of equicardinal maps. Numbers on arcs of the graph give the weights. The blue arc are points of weight 2, the red arcs are points of weight 1.\label{WnotUfig}}
\end{figure}
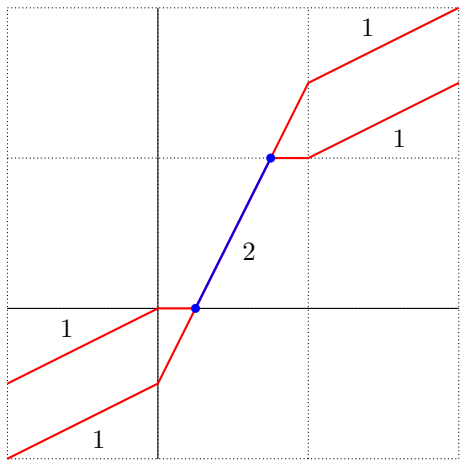

In this case, $\varphi$ is an $\N$-weighted map with weighted index 2. If $\varphi$ were a union of equicardinal maps, it must be globally expressible as either equicardinal (2-valued) itself, or as a union of two single-valued maps. Inspection of the graph shows that neither case is possible, since on the domain $[0,1]$ the map is a union of two single-valued maps, but on the domains $[-1,0]$ and $[1,2]$ it is not a union of two single-valued maps.
\end{ex}

The following Theorem and its proof will appear in Section \ref{weightedequcardinalsection}, demonstrating a large class of space for which every weighted map is homotopic to a union of single-valued maps.
\begin{thm}\label{weighted-equicardinal}
If $Y$ is $n$-connected and $\dim X\leq 2n$, then any $\N$-weighted map $\varphi: X\to Y$ is $w$-homotopic to one of the form $\sum_{i=1}^k \lambda_i f_i$, where $\lambda_i\in \Q\setminus\{0\}$ and $f_i: X\to Y$ is continuous, for all $1\leq i\leq k$.
\end{thm}

%
%

\section{$\Psi(\mathcal W)\subsetneq \mathcal C$}
From our basic definitions it is clear that $\Psi(\mathcal W)\subseteq \mathcal C$. In this section we give an example map showing that these two classes are not equal. That is, the following example is an at-most-$n$-valued map which cannot be expressed as an $\N$-weighted map.

\begin{ex}\label{forkmap}
Viewing the circle $S^1$ as the interval $[0,1]$ with endpoints identified, let $f:S^1 \multimap S^1$ be the at-most-3-valued map given in Figure \ref{forkfig}. We will argue that this $f$ cannot be expressed as any $\N$-weighted map.

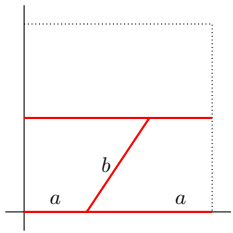
\begin{figure}
\[
\begin{tikzpicture}[scale=2.5]
\draw (-.1,0) -- (1.1,0);
\draw (0,-.1) -- (0,1.1);
\draw[densely dotted] (0,1) -- (1,1) -- (1,0);

\draw[red,thick] (0,0) -- (1,0);
\draw[red,thick] (0,.5) -- (1,.5);
\draw[red,thick] (.333,0) -- (.666,.5);

\tikzstyle{weightnode}=[scale=.8]
\node[weightnode,above] at (.166,0) {$a$};
\node[weightnode,left] at (.5,.25) {$b$};

\node[weightnode,above] at (.833,0) {$a$};
\end{tikzpicture}
\]
\caption{An at-most-3-valued map which is not an $\N$-weighted map. The $x$-coordinates of the two branching points are $\frac13$ and $\frac23$. Arcs of the graph are labeled according to the argument given in Example \ref{forkmap}.\label{forkfig}}
\end{figure}

To obtain a contradiction, assume that $f$ is an $\N$-weighted map. By continuity of the weighted map, the weights at the values along the bottom arc of the graph must be constant. That is, there must be some $a>0$ such that $w_f(x,0) = a$ for all $x\neq \frac13$. Similarly, there must be some $b>0$ such that $w_f(x,y)=b$ when $(x,y)$ is in the diagonal arc of the graph. (See labels in Figure \ref{forkfig}.)

But continuity of the weighted map $f$ requires that the weights sum at the point $(\frac13, 0)$, and so we must have $a=b+a$, which is impossible since $a$ and $b$ are greater than 0.
\end{ex}

The example above shows that the class of $\N$-weighted maps (and thus also symmetric product maps) is strictly smaller than the class of all general at-most-$n$-valued maps. (Recall by Proposition \ref{samen=2} that these classes are the same in the case of $n=2$.)

\section{The class of $\{1,n\}$-valued maps}
In \cite{Oneill}, O'Neill developed a fixed point theory for the following class of multivalued selfmaps: given a space $X$ and a natural number $n$, O'Neill considers maps $f:X\multimap X$ for which $f(x)$ is a set consisting of 1 or $n$ acyclic components. Such maps are shown to have a well-defined induced homomorphism in homology, and satisfy a Lefschetz fixed point theorem.

Following some discussion in \cite{schirmer}, we make the following definition: Given a natural number $n$, we say an at-most-$n$-valued map $f:X\multimap Y$ is a \emph{$\{1,n\}$-valued map} if the cardinality of $f(x)$ is 1 or $n$ for each $x$. (Schirmer's paper \cite{schirmer} concerns $\{1,2\}$-valued maps, which she calls ``bimaps.'')

Given spaces $X$ and $Y$, let $\mathcal D_n(X,Y)$ be the set of $\{1,n\}$-valued maps from $X$ to $Y$. In this Section we  briefly discuss how this set fits into the containments described in \eqref{containments}. The results are that ${\mathcal D}_n(X,Y) \subset \Psi_{\mathcal S}(\mathcal S) = \Psi_{\mathcal W}(\mathcal W)$, but ${\mathcal D}_n(X,Y)$ is not comparable to $\Psi(\mathcal U)$ or $\Psi(\mathcal F)$.

\begin{thm}
${\mathcal D}_n(X,Y) \subset \Psi_{\mathcal S}(\mathcal S)$.
\end{thm}
\begin{proof}
Let $f\in {\mathcal D}_n(X,Y)$. Since $f:X\multimap Y$, we can view $f$ as a single-valued map $f:X\to C_n(Y)$. Let $A\subset C_n(Y)$ be the collection of all sets consisting of either $1$ or $n$ elements, and there is a continuous bijection $i:A \to \SP^n(Y)$ given by $i(\{y\}) = [y,\dots,y]$ and $i(\{y_1,\dots,y_n\}) = [y_1,\dots,y_n]$. Then $g = i \circ f: X \to \SP^n(Y)$ is a symmetric product map, and we have $\Psi(g) = f$ and so $f \in \Psi(\mathcal S)$ as desired.
\end{proof}

Thus every $\{1,n\}$-valued map can be expressed naturally as a symmetric product map, though not all symmetric product maps are $\{1,n\}$-valued when $n>2$. For example the map of Figure \ref{123valuedfig} is an at-most-3-valued map on $[0,1]$ for which the cardinality of $f(x)$ is one of $\{1,2,3\}$. This map is a union of equicardinals, and so corresponds to a $3$-fold map and a weighted map and a symmetric product map, but it is not a $\{1,3\}$-valued map.
\begin{figure}
\[
\begin{tikzpicture}[scale=2.5]
\draw (-.1,0) -- (1.1,0);
\draw (0,-.1) -- (0,1.1);
\draw[densely dotted] (0,1) -- (1,1) -- (1,0);

\draw[red,thick] (1,1) -- (0,0) -- (1,0);
\draw[red,thick] (.5,.5) -- (1,.5);

\end{tikzpicture}
\]

\caption{A union of equicardinal maps on the interval $[0,1]$ which is not a $\{1,n\}$-valued map.\label{123valuedfig}}
\end{figure}
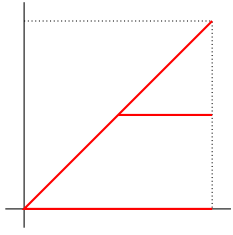

\section{A splitting result for $\Q$-weighted maps}\label{weightedequcardinalsection}

In this section we consider a more general class of weighted maps with weights in the rational numbers $\mathbb Q$. Our main result for the section proves in a large class of cases that a $\mathbb Q$-weighted map will split as a weighted sum of single-valued maps.

Let \( R \) be a commutative ring with a unit $1$. We denote by \( R(X) \) the free \( R \)-module generated by the set \( X \). Identifying \( X \) with the canonical basis of \( R(X) \), any element \( \zeta \in R(X) \) can be expressed in the form
\begin{equation}\label{xi}
\zeta = \sum_{x\in X} \zeta_x \cdot x,
\end{equation}
where \( \zeta_x \) is called the coefficient of \( \zeta \) corresponding to the point \( x \in X \). Sometimes the symbol $\zeta_x$ will be denoted as follows
\begin{equation}\label{xi-2}
\zeta_x=\langle \zeta, x\rangle.
\end{equation}
Given $\zeta\in R(X)$, the support of \( \zeta \) is defined as
\begin{equation*}
\text{supp}(\zeta):=\{x\in X \mid \langle \zeta, x \rangle \neq 0 \}.
\end{equation*}
For a subset \( A \subseteq X \) and $\zeta\in R(X)$, we denote by \( m_A(\zeta) \) the sum of the coefficients of \( \zeta \) at the points in \( A \), i.e.,
\begin{equation*}
m_A(\zeta):=\sum_{x\in A}\langle \zeta,x\rangle\in R.
\end{equation*}

\begin{defn}\label{weighted map}
Let \( X \) and \( Y \) be topological Hausdorff spaces. A weighted map
\[
\phi \colon X \to Y
\]
with coefficients in \( R \) is a homomorphism \( \phi \colon R(X) \to R(Y) \) such that there exists a multivalued upper semi-continuous map
\[
t_{\phi} \colon X \multimap Y
\]
satisfying the following conditions:
\begin{enumerate}
\item[(1)] For every \( x \in X \), \( t_{\phi}(x) \) is a finite subset of \( Y \).
\item[(2)] The support of \( \phi(x)\in R(Y) \) is contained within \( t_{\phi}(x) \), i.e.,
\begin{equation*}
\text{supp}(\phi(x)) := \{y \in Y \mid \langle \phi(x),y\rangle \neq 0\} \subset t_{\phi}(x).
\end{equation*}
\item[(3)] If \( U \) is an open subset of \( Y \) and \( x \in X \) is such that \( t_{\phi}(x) \cap \partial U = \emptyset \), then
\[
m_U(\phi(x)) = m_U(\phi(\tilde{x})),
\]
for any \( \tilde{x} \in X \) sufficiently close to \( x \).
\end{enumerate}
\end{defn}

\begin{rem}\label{two-definitions-of-weighted}
Definition \ref{definitionofweightedmap1309} was originally stated for the set of natural numbers \( \mathbb{N} \). However, it is evident that the definition also holds for any ring \( R \).
Thus, if we replace \( \mathbb{N} \) with an arbitrary ring \( R \), Definition \ref{definitionofweightedmap1309} becomes equivalent to Definition
\ref{weighted map}, as we demonstrate below.

By Definition \ref{weighted map}, a weighted mapping is a pair
\[
(t_{\phi}: X \multimap Y, \phi: R(X) \to R(Y)),
\]
which, for simplicity, we denote as \(\phi: X \to Y\). Observe that a pair
\[
(t_{\phi}:X\multimap Y,\phi:R(X)\to R(Y))
\]
induces the following pair
\[
(n_{\phi}:X\multimap Y,w_{\phi}:X\times Y\to R)
\]
in the following way:
\begin{align*}
&n_{\phi}(x):=t_{\phi}(x)\text{ for all }x\in X,\\
&w_{\phi}(x,y):=\xi_y^x\text{ for all }x\in X\text{ and }y\in Y,
\end{align*}
where $\phi: R(X)\to R(Y)$ is given by
\begin{align*}
\phi(x)=\sum_{y\in Y}\xi_y^x y \in R(Y)\quad \text{ (comp. \eqref{xi})}.
\end{align*}
Conversely, a pair
\[
\phi=(n_{\phi}:X\multimap Y,w_{\phi}:X\times Y\to R)
\]
induces a pair
\[
(t_{\phi}:X\multimap Y,\phi: R(X)\to R(Y))
\]
as follows
\begin{align*}
&t_{\phi}(x):=n_{\phi}(x)\text{ for all }x\in X,\\
&\phi(x):=\sum_{y\in Y}w(x,y)y\text{ for all }x\in X\text{ and }y\in Y.
\end{align*}
In summary, since the set of natural numbers \( \mathbb{N} \) is a subset of the rationals \( \mathbb{Q} \), weighted maps with coefficients in \( \mathbb{N} \) can be analyzed using homotopy invariants associated with the set of rational numbers.
\end{rem}
The multivalued map \( t_{\phi} \) in Definition \ref{weighted map} is called an upper semi-continuous support of \( \phi \) (usc-support). Among all upper semi-continuous supports of \( \phi \), there exists a minimal one, denoted by \( t_0 \), defined as:
\[
t_{0}(x) = \bigcap_{t_{\phi} \in \text{supp}(\phi)} t_{\phi}(x),
\]
where \( \text{supp}(\phi) \) denotes the set of all usc-supports of \( \phi \).
It is straightforward to verify that the graph of \( t_{0} \) is the closure of the set
\[
\{ (x, y) \mid \langle \phi(x), y \rangle \neq 0 \}.
\]

From property $(3)$ of Definition \ref{weighted map}, it follows that for any weighted map \( \phi \colon X \to Y \), the function \( x \mapsto m_Y(\phi(x)) \) is locally constant. Consequently, if \( X \) is connected, it is constant on  \( X \), and  the value
\[
I_w(\phi) := m_Y(\phi(x))
\]
is called the index of \( \phi \).


\begin{rem}\label{operation-weighte-maps}
The following operations can be performed on weighted maps:
\begin{itemize}

\item  Any continuous map \( f \colon X \to Y \) can be regarded as a weighted map in the following way. Specifically,
\( f: R(X) \to R(Y) \) is defined by
\[
f(\zeta) = f\left(\sum_{x\in X} \langle\zeta,x \rangle x \right) = \sum_{x\in X} \langle \zeta,x\rangle  f(x), \quad \text{for all } \zeta \in R(X).
\]
Additionally, \( t_f \colon X \to Y \) is defined by \( t_f(x) = f(x) \) for all \( x \in X \).

\item Weighted maps can be added and scaled by elements of \( R \). Given two weighted maps \( \phi \colon X \to Y \) and \( \psi \colon X \to Y \), their sum \( \phi + \psi \colon X \to Y \) is also a weighted map. Indeed, the upper semi-continuous (usc) map \( t \) defined by \( t(x) = t_\phi(x) \cup t_\psi(x) \) is a usc-support for the homomorphism \( \phi + \psi \colon R(X) \to R(Y) \). In particular, \( t_{\phi + \psi} \subseteq t_\phi \cup t_\psi \). Furthermore, if \( \phi \) is a weighted map and \( r \in R \), then \( r\phi \) is also a weighted map, where $(r\phi)(x)=r\phi(x)$. Hence, any linear combination of a finite number of continuous maps with coefficients in \( R \) is a weighted map.

\item The composition of two weighted maps is a weighted map. Specifically, if \( t \) is a support of \( \phi \colon X \to Y \), and \( t' \) is a support of \( \psi \colon Y \to Z \), then the upper semi-continuous map \( t't \) is a support of the homomorphism \( \psi\phi \colon R(X) \to R(Z) \). The set of all weighted maps from \( X \) to \( Y \) naturally forms an \( R \)-module, and this structure is preserved under composition.

\item The product of two weighted maps \( \phi \colon X \to Y \) and \( \phi' \colon X' \to Y' \) is the weighted map \( \phi \times \phi' \colon X \times X' \to Y \times Y' \), defined by
\[
(x, x') \mapsto \sum_{y \in t_\phi(x), \; y' \in t_{\phi'}(x')} \langle \phi(x), y \rangle \langle \phi'(x'), y' \rangle \, (y, y').
\]

\end{itemize}
\end{rem}
Let \( \mathcal{CW} \) denote the category of CW-complexes and continuous maps. Given \( X \in \mathcal{CW} \) and \( x_0 \in X \), we will refer to the pair \( (X, x_0) \) as a {\bf pointed space} and \( x_0 \) as the {\bf base point} of \( X \). Similarly, a continuous map \( f \colon X \to Y \) such that \( f(x_0) = y_0 \) will be called a {\bf pointed map} and will be denoted by \( f \colon (X, x_0) \to (Y, y_0) \). Note that the base point can be chosen arbitrarily. The category of pointed CW-complexes and pointed maps between them will be denoted by \( \mathcal{CW}_0 \).
A pointed homotopy \( h \colon (X, x_0) \times [0,1] \to (Y, y_0) \) is a map such that for each \( t \in [0,1] \), the partial map \( h_t \colon (X, x_0) \to (Y, y_0) \), defined by \( h_t(x) = h(x, t) \), is a pointed map.

We denote the morphisms from \( X \) to \( Y \) in the category \( \mathcal{CW} \) by \( [X; Y] \). Similarly, we denote by \( [(X,x_0);(Y,y_0)] \) (respectively, \( \pi[(X,x_0); (Y,y_0)] \)) the set of all pointed morphisms from \( (X, x_0) \) to \( (Y, y_0) \) in the category \( \mathcal{CW}_0 \) (respectively, the set of all pointed homotopy classes of morphisms from \( (X, x_0) \) to \( (Y, y_0) \) in the homotopy category \( \pi \mathcal{CW}_0 \)).

The weighted category over the ring \( R \) is a category that has the same objects as \( \mathcal{CW} \), but its morphisms are weighted maps {\bf of index $0$} with coefficients in \( R \). This category is denoted by \( w\text{-}\mathcal{CW} \). We denote the morphisms from \( X \) to \( Y \) in \( w\text{-}\mathcal{CW}\) by \( [X; Y]^w \).

To define the category \( w\text{-}\mathcal{CW}_0 \), we first introduce the following concept: a \( w \)-map \( \psi \colon X \to Y \) of index zero  over a ring \( R \) is called {\bf pointed} if its minimal support, denoted \( t_{\psi} \), satisfies \( t_{\psi}(x_0) = y_0 \), where \( x_0 \in X \) and \( y_0 \in Y \) are the base points. Such a weighted map will be denoted by \( \psi \colon (X, x_0) \to (Y, y_0) \). The category \( w\text{-}\mathcal{CW}_0 \) consists of the same objects as \( \mathcal{CW}_0 \), but its morphisms are pointed weighted maps. The morphisms from \( (X,x_0) \) to \((Y,y_0) \) in \( w\text{-}\mathcal{CW}_0 \) will be denoted by \( [(X,x_0); (Y,y_0)]^w\).

Given two pointed weighted maps \( \phi \) and \( \psi \) from \( (X, x_0) \) to \( (Y, y_0) \) of index zero, we say that \( \phi \) is \( w \)-homotopic to \( \psi \) if there exists a weighted map \( \theta \colon (X, x_0) \times [0,1] \to (Y, y_0) \) such that \( \theta_t = \theta(\cdot, t) \) is a pointed weighted map for each \( t \in [0, 1] \), with the following properties:
\[
\theta(x, 0) = \phi(x) \quad \text{and} \quad \theta(x, 1) = \psi(x).
\]

Thus, the above concept allows us to introduce the category \( \pi w\text{-}\mathcal{CW}_0 \), whose objects are the same as those in \( w\text{-}\mathcal{CW}_0 \), and whose morphisms are the pointed \( w \)-homotopy classes of all morphisms from the category \( w\text{-}\mathcal{CW}_0 \).  The set of all morphisms from $(X,x_0)$ to $(Y,y_0)$ in this category will be denoted by
\[
\pi^w[(X,x_0); (Y,y_0)].
\]

By forgetting the base point in the category \( \pi w\text{-}\mathcal{CW}_0 \), we obtain the category denoted by \( \pi w\text{-}\mathcal{CW} \).

\begin{lem}
The set $\pi^w[(X,x_0); (Y,y_0)]$ has the structure of a (left) $R$-module, where
\begin{itemize}
\item $+:\pi^w[(S^n,s_0); (X,x_0)]\times \pi^w[(S^n,s_0); (X,x_0)]\to\pi^w [(S^n,s_0); (X,x_0)]$,
\item $\cdot : R\times \pi^w[(S^n,s_0); (X,x_0)]\to\pi^w [(S^n,s_0); (X,x_0)]$
\end{itemize}
 are defined as in Remark \ref{operation-weighte-maps}.
\end{lem}
\begin{rem}
From now on, we assume that \( R = \mathbb{Q} \). All spaces considered are assumed to be connected CW-complexes.
\end{rem}
We will now describe the $w$-homotopy theory for weighted maps.

\begin{defn}
The $n$-th \(w\)-homotopy module \(\pi_n^w(X,x_0)\) of a pointed space \((X,x_0)\) is defined as the module \((\pi^w[(S^n,s_0); (X,x_0)],+,\cdot)\).
\end{defn}

Given two pointed compact spaces \((X, x_0)\) and \((Y, y_0)\), we define the smash product as follows:
\[
X \wedge Y = (X \times Y) / \big(X \times \{y_0\} \cup \{x_0\} \times Y\big),
\]
where the sets \(X \times \{y_0\}\) and \(\{x_0\} \times Y\) are collapsed to a single equivalence class. Thus, the basepoint of the pointed space
\(X \wedge Y\) is the equivalence class of \((x_0, y_0)\), denoted by \(x_0 \wedge y_0\). Each element of the space $X\wedge Y$ is denoted by the symbol $x\wedge y$. The smash product for compact spaces has the following properties, where $\cong$ denotes homeomorphism:
\begin{enumerate}
\item[(a)] $X\wedge Y \cong Y\wedge X$.
\item[(b)] $(X\wedge Y)\wedge Z\cong X\wedge (Y\wedge Z)$.
\item[(c)] $S^n\wedge S^m\cong S^{n+m}$ and $S^0\wedge X\cong X$.
\item[(d)] $S^n\wedge (S^m\wedge X)\cong (S^n\wedge S^m)\wedge X\cong S^{n+m}\wedge X$.
\item[(e)] $S^1\wedge X\cong \Sigma X$, where $\Sigma X:=X\times [0,1]/(X\times\{0\}\cup X\times \{1\}\cup \{x_0\}\times [0,1])$
in the literature is called the reduced suspension of  $X$.
\item[(f)] The smash product $X\wedge Y$ of two CW-complexes $(X,x_0)$ and $(Y,y_0)$ is also a CW-complex.
\end{enumerate}

\begin{lem}(\cite[Corollaries 5.15 and 5.22]{whitehead} or \cite{Pejs1})
If \( X \) and \( Y \) are CW-complexes, then the following set
\[
\pi\big[(S^n \wedge X, s_0 \wedge x_0); (Y, y_0)\big]
\]
has the structure of a group for \( n \geq 1 \) and an abelian group for \( n \geq 2 \), where the operation
\[
\ast: \pi[S^n\wedge X; Y] \times \pi[S^n\wedge X; Y] \to \pi[S^n\wedge X; Y]
\]
(base points omitted) is defined as follows:
\[
(f \ast g)([z,u]) =
\begin{cases}
f([z,2u]) & \text{if } 0 \leq u \leq \frac{1}{2}, \\
g([z,2u - 1]) & \text{if } \frac{1}{2} \leq u \leq 1,
\end{cases}
\]
where we use the identification
\[
S^n \wedge X = S^1 \wedge (S^{n-1} \wedge X) = \Sigma(S^{n-1} \wedge X),
\]
with \( z \in S^{n-1} \wedge X \) and \( u \in [0,1] \). Recall that \( S^0 \wedge X \cong X \) for any space \( X \).
\end{lem}


A pointed map \(\varphi : (X, x_0) \to (Y, y_0)\) induces a map on the smash product, \(\id_{S^1} \wedge \varphi: S^1 \wedge X \to S^1 \wedge Y\), defined as the map induced by the quotient of \(\id_{S^1} \times \varphi: S^1 \times X \to S^1 \times Y\).  More precisely, if $\varphi(x)=\sum \lambda_i y_i$, then
\[
(\id_{S^1} \wedge \varphi)(s\wedge x)=\sum \lambda_i (s\wedge y_i) \text{ for }s\in S^1,x\in X.
\]


Now, given two spaces $Z,W$, we will define two mappings, \( S \) and \( S^w \), referred to as suspension mappings.
Specifically, we define
\[ S: \pi[(Z,z_0);(W,w_0)]\to \pi[(S^1\wedge Z,s_0\wedge z_0);(S^1\wedge W,s_0\wedge w_0)] \]
and
\[ S^w: \pi^w[(Z,z_0);(W,w_0)]\to \pi^w[(S^1\wedge Z,s_0\wedge z_0);(S^1\wedge W,s_0\wedge w_0)]  \]
by
\[
S[f]:=[\id_{S^1}\wedge f] \text{ and } S^w[\varphi]:=[\id_{S^1}\wedge\varphi],
\]
respectively.\\

\begin{rem}
From this point onward, we will omit base points in the notation.
\end{rem}


Recall the following extended version of the  Freudenthal suspension theorem:
\begin{lem}(\cite[Chapter XII]{whitehead} and \cite{Pejs1, Pejs2})
Let \(Y\) be an \(m\)-connected CW-complex and \(X\) a CW-complex of dimension \(\leq 2m\). Then the suspension map
\begin{equation*}
S: \pi[X;Y]\to \pi[S^{1} \wedge X; S^{1} \wedge Y]
\end{equation*}
is an isomorphism.  In particular, the group structure of $\pi[S^1 \wedge X; S^1 \wedge Y]$ can be carried by $S$ in such a way that $S$ becomes a group homomorphism.
\end{lem}
From the above lemma it follows immediately the following corollary:
\begin{cor}(\cite[Chapter XII]{whitehead})
Let \(Y\) be an \(m\)-connected CW-complex and \(X\) a CW-complex of dimension \(\leq 2m\). Then the suspension map
\begin{equation*}
S: \pi[S^n \wedge X; S^n \wedge Y]\to \pi[S^{n+1} \wedge X; S^{n+1} \wedge Y] \text{ given by }S[f]=[\id_{S^1} \wedge f],
\end{equation*}
for all $n\geq 1$, is an isomorphism. In particular, $S$ is a group homomorphism for all $n\geq 1$,
\end{cor}


In the case of weighted maps, the above lemma and corollary hold with weaker assumptions.
\begin{lem}(\cite{Pejs1,Pejs2})\label{weighted-suspension}
Let \(X\) and \(Y\) be CW-complexes. Then the suspension map $S^w$
\begin{equation*}
S^w: \pi^w[S^n \wedge X; S^n \wedge Y]\to \pi^w[S^{n+1} \wedge X; S^{n+1} \wedge Y], \text{ for all }n\geq 0,
\end{equation*}
is an isomorphism. In particular, \( S^w \) is a \( \mathbb{Q} \)-homomorphism for all $n\geq 0$ (recall that $S^0\wedge X \cong X$ for any space $X$).
\end{lem}
Since the smash operation  induces a map of $\pi[X;Y]$ into $\pi[S^1\wedge X;S^1\wedge Y]$, we can iterate
the procedure to obtain an infinite sequence of homomorphisms
\begin{multline*}
\pi[X;Y]\xrightarrow{S} \pi[S^1\wedge X;S^1\wedge Y]\xrightarrow{S} \pi[S^2\wedge X;S^2\wedge Y]\xrightarrow{S}\cdots \\
\xrightarrow{S} \pi[S^n\wedge X;S^n\wedge Y]\xrightarrow{S} \cdots.
\end{multline*}
Thus we obtain the direct limit
\begin{equation*}
\{X,Y\}:=\lim_{\longrightarrow}\pi[S^n\wedge X;S^n\wedge Y]
\end{equation*}
along with a family of homomorphisms:
\begin{equation*}
\{h_n: \pi[S^n\wedge X;S^n\wedge Y]\to \{X,Y\}\mid n\in \mathbb{N}\}
\end{equation*}
which satisfy the commutative diagram:
\begin{align*}
\xymatrix@C=25pt@R=25pt{
\pi[S^n\wedge X;S^n\wedge Y]\ar[r]^-{h_n}\ar[d]_{S} &  \{X,Y\}
\\
\pi[S^{n+1}\wedge X;S^{n+1}\wedge Y]\ar[ur]_{h_{n+1}}. &
 }
\end{align*}
Similarly as above, we obtain an infinite sequence for weighted maps:
\begin{multline*}
\pi^w[X;Y]\to \pi^w[S^1\wedge X;S^1\wedge Y]\to \pi^w[S^2\wedge X;S^2\wedge Y]\to\cdots \\
\to \pi^w[S^n\wedge X;S^n\wedge Y]\to \cdots
\end{multline*}
and the direct limit
\begin{equation*}
\{X,Y\}^w:=\lim_{\longrightarrow}\pi^w[S^n\wedge X;S^n\wedge Y]
\end{equation*}
together with a family of homomorphisms:
\begin{equation*}
\{h_n^w: \pi^w[S^n\wedge X;S^n\wedge Y]\to \{X,Y\}^w\mid n\in \mathbb{N}\}
\end{equation*}
which satisfy the commutative diagram:
\begin{align*}
\xymatrix@C=25pt@R=25pt{
\pi^w[S^n\wedge X;S^n\wedge Y]\ar[r]^-{h_n^w}\ar[d]_{S^{w}} &  \{X,Y\}^w
\\
\pi^w[S^{n+1}\wedge X;S^{n+1}\wedge Y]\ar[ur]_{h_{n+1}^w}. &
 }
\end{align*}
However, in this case Lemma \ref{weighted-suspension} implies that

\begin{equation}\label{iso}
\{X,Y\}^w:=\lim_{\longrightarrow}\pi^w[S^n\wedge X;S^n\wedge Y]\cong\pi^w[X;Y],
\end{equation}
and hence
\begin{equation*}
h_n^w: \pi^w[S^n\wedge X;S^n\wedge Y]\xrightarrow{\cong} \{X,Y\}^w\text{ for all }n\in \mathbb{N}.
\end{equation*}
For any space \((X,x_0) \),  \((Y,y_0) \) and $(S^n,s_0^n)$, consider a map
\begin{equation}\label{J-n}
J_n: \pi[S^n\wedge X;S^n\wedge Y]\to \pi^w [S^n\wedge X;S^n\wedge Y]
\end{equation}
given by
\begin{equation*}
J_n[f]=[f-s_0^n\wedge y_0],
\end{equation*}
where $s_0^n\wedge y_0$ represents the constant map.




\begin{lem}(\cite{Pejs1,Pejs2})
If \( X \) and \( Y \) are CW-complexes and \( n \geq 1 \), then \( J_n \) is a homomorphism. If \( n = 0 \), then \( J_0 \) is a homomorphism under the assumption that \( Y \) is an \( m \)-connected CW-complex and \( X \) is a CW-complex of dimension at most \( 2m \).
\end{lem}

Taking a direct limit in  the system:
\begin{equation*}
\{J_n\}: \{\pi[S^n\wedge X;S^n\wedge Y]\}\to \{\pi^w[S^n\wedge X;S^n\wedge Y]\},
\end{equation*}
we obtain a limiting homomorphism:
\begin{equation*}
J_{\infty}: \{X,Y\}\to \{X,Y\}^w
\end{equation*}
and a commutative diagram:
\begin{align*}
\xymatrix@C=35pt@R=25pt{
\pi[S^n\wedge X;S^n\wedge Y]\ar[r]^-{J_n}\ar[d]_{h_n} & \ar[d]_{h_n^w}^{\cong} \pi^w[S^n\wedge X;S^n\wedge Y]
\\
\{X,Y\}\ar[r]^-{J_{\infty}} & \{X,Y\}^w.
 }
\end{align*}
Now we modify the above diagram by using the tensor product:
\begin{align}\label{diagram-h-n}
\begin{split}
\xymatrix@C=35pt@R=25pt{
\pi[S^n\wedge X;S^n\wedge Y]\otimes \mathbb{Q}\ar[r]^-{J_n\otimes \id_{\mathbb{Q}}}\ar[d]_{h_n\otimes \id_{\mathbb{Q}}} & \ar[d]_{h_n^w\otimes \id_{\mathbb{Q}}}^{\cong}
\pi^w[S^n\wedge X;S^n\wedge Y]\otimes \mathbb{Q}
\\
\{X,Y\}\otimes \mathbb{Q}\ar[r]^-{J_{\infty}\otimes \id_{\mathbb{Q}}} & \{X,Y\}^w\otimes \mathbb{Q},
 }
 \end{split}
\end{align}
for all $n\geq 0$. We will need the following lemma:
\begin{lem}(\cite{Pejs2})
$J_{\infty}\otimes \id_{\mathbb{Q}}: \{X,Y\}\otimes \mathbb{Q}\to \{X,Y\}^w\otimes \mathbb{Q} $ is an isomorphism.
\end{lem}

Now we are ready to prove the main lemma of this section:
\begin{lem}\label{main-result}
If $Y$ is $m$-connected and $\dim X\leq 2m$, then
\[
J_{0}\otimes \id_{\mathbb{Q}}: \pi [X;Y]\otimes \mathbb{Q}\to \pi^w[X;Y]\otimes \mathbb{Q}=\pi^w[X;Y]
\]
is an isomorphism.
\end{lem}
\begin{proof}
At the beginning of the proof, we present the following observations:
\begin{itemize}
\item $\pi^w[X;Y]\otimes \mathbb{Q}=\pi^w[X;Y]$ because $\pi^w[X;Y]$ is a $\mathbb{Q}$-module.
\item The Freudenthal suspension theorem implies that $h_0$ from \eqref{diagram-h-n} is an isomorphism.
\end{itemize}
Thus, for $n=0$, the diagram \eqref{diagram-h-n}  takes the following form:
\begin{align}
\begin{split}
\xymatrix@C=35pt@R=25pt{
\pi[X;Y]\otimes \mathbb{Q}\ar[r]^-{J_0\otimes \id_{\mathbb{Q}}}\ar[d]_{h_0\otimes \id_{\mathbb{Q}}}^{\cong} & \ar[d]_{h_0^w\otimes \id_{\mathbb{Q}}}^{\cong}
\pi^w[X;Y]\otimes \mathbb{Q}
\\
\{X,Y\}\otimes \mathbb{Q}\ar[r]^-{J_{\infty}\otimes \id_{\mathbb{Q}}}_-{\cong} & \{X;Y\}^w\otimes \mathbb{Q},
 }
 \end{split}
\end{align}
which, in turn, implies the desired conclusion.
\end{proof}

The next lemma says that any weighted map can be deformed to a pointed weighted map.

\begin{lem}(\cite{Pejs1})\label{homotopy}
Any weighted map between two pointed and connected CW-complexes is $w$-homotopic (freely) to a pointed weighted map.
\end{lem}

Now we are ready to prove that a weighted map, under certain conditions on the domain and codomain, can be deformed to a $\Q$-weighted sum of single-valued maps.

\begin{cor}\label{main-result-2}
If $Y$ is $m$-connected and $\dim X\le 2m$, then any weighted map $\varphi: X\to Y$ with weights in $\mathbb Q$ is $w$-homotopic to one of the form $\sum_{i=1}^k \lambda_i f_i$, where $\lambda_i\in \mathbb{Q}\setminus\{0\}$ and $f_i: X\to Y$ is continuous for each $1\leq i\leq k$.
\end{cor}
\begin{proof}
By Lemma \ref{homotopy}, we may assume that $\phi$ is a pointed map.

First we prove the special case in which $I_w(\phi)=0$. Since the weighted index is zero, this means that $\phi$ corresponds to an element of $[X;Y]^w$. By Lemma \ref{main-result}, $\phi$ belongs to a homotopy class in the image of $J_0 \otimes \id_\Q$. Since
\[
(J_0\otimes \id_{\mathbb{Q}})([f]\otimes q)=[q(f-y_0)],
\]
we obtain the desired decomposition into a sum of single-valued maps.

Now we show that the general case for $I_w(\phi)\neq 0$ can be reduced to the index zero case. Let $\psi:X\to Y$ be given by
\[ \psi = \phi - I_w(\phi)y_0, \]
and then we will have $I_w(\psi)=0$. Thus by the first part, $\psi$ is $w$-homotopic to a weighted sum of single-valued maps. But we have
\[ \phi = \psi + I_w(\phi)y_0, \]
and so $\phi$ is $w$-homotopic to a weighted sum of single-valued maps as desired.
\end{proof}

\begin{cor}
Under the above assumptions, if \(\varphi: X \to Y\) is an \(\mathbb{N}\)-weighted map, then there exists a number \(k_0 \in \mathbb{N}\) such that \(k_0\varphi: X \to Y\) is \(w\)-homotopic to a map of the form \(\sum_{i=1}^k \beta_i f_i\), where \(\beta_i \in \mathbb{Z} \setminus \{0\}\) and \(f_i: X \to Y\) is continuous for all \(1 \leq i \leq k\).
\end{cor}
\begin{proof}
From the above corollary it follows that $\varphi: X\to Y$ is $w$-homotopic to to one of the form $\sum_{i=1}^k \lambda_i f_i$, where $\lambda_i\in \mathbb{Q}\setminus\{0\}$ and $f_i: X\to Y$ is continuous, for all $1\leq i\leq k$. Let
\begin{equation*}
\lambda_i=\frac{m_i}{n_i}
\end{equation*}
where $m_i\in \mathbb{Z}$ and $n_i\in \mathbb{N}$, for $1\leq i\leq k$. Now we put
\begin{equation*}
k_0:=n_1\cdot n_2\cdot\ldots\cdot n_k.
\end{equation*}
Then $k_0\varphi$ is $w$-homotopic to $\sum_{i=1}^k \beta_i f_i$, where $\beta_i:=k_0\cdot \lambda_i\in \mathbb{Z}$.
\end{proof}

Recall that our classification of various classes of at-most-$n$-valued maps has shown that the class of $\N$-weighted maps is the same as the class of symmetric product maps. The larger class of weighted maps with weights in $\Q$ will include some maps which are not symmetric product maps. For example the map of Figure \ref{forkfig} can be made a weighted map by assigning weights such that $b=0$ in that figure. It is natural to ask which general at-most-$n$-valued maps can be made into weighted maps with weights in $\Q$.

\begin{open}
Let \( X \) and \( Y \) be metric spaces, and let \( f: X \multimap Y \) be an at-most-\( n \)-valued map. Is there some nontrivial weight function with weights in $\Q$ which makes $f$ into a weighted map?
\end{open}

\section{Topology of the at-most-$n$-valued configuration space of the circle}
In this section we show that $C_n(S^1)$ has the homotopy type of $S^n$ when $n$ is odd, and $S^{n-1}$ when $n$ is even.

Recall that $C_n(S^1)$ is defined as a quotient of the cartesian product $(S^1)^n$. For this section, it will be convenient to consider $S^1$ as the interval $[0,1]$ with endpoints identified, and $C_n(S^1)$ as a quotient of $[0,1]^n$ where we additionally identify $0$ and $1$.

Let $\Delta_n$ be the $n$-simplex given by vertices as $\Delta_n = \langle b_0, \dots, b_n\rangle$, where $b_k = e_1 + \dots + e_k$, the $e_i's$ are the standard basis vectors
for $\R^n$, and $b_0$ is the origin of  $\R^n$. In coordinates we have   $b_0=(0,\dots,0) \  b_1=(1,0,\dots,0), b_2=(1,1,0,\dots,0), \dots, b_n=(1,1,\dots,1)$. Given an
 element $w$ of $[0,1]^n$ there is an element $w_0\in \Delta_n$ such that they define the same element  in  $C_n(S^1)$. Furthermore, any two points of the interior of $\Delta_n$ will represent two different points of $C_n(S^1)$.

Then $C_n(S^1)$  is obtained  as the quotient of $\Delta_n$ by a relation on the boundary of $\Delta_n$, where certain faces of the boundary are identified. Let  $q:\Delta_n \to C_n(S^1)$ be the quotient map which identifies the appropriate faces.

For $k\le n$, we denote a $k$-face of $\Delta_n$ by $\langle b_{i_1},\dots,b_{i_k}\rangle$ for distinct $i_j\in \{0,\dots,n\}$. We say that a $k$-face $\sigma \subset \Delta_n$ is \emph{extremal} if it contains both $b_0$ and $b_n$. Otherwise, it is \emph{non-extremal}.

\begin{lem}\label{faceslemma}
When $k$-faces of $\Delta_n$ are identified via $q$, there are two classes: the extremal faces, and the non-extremal faces.
\end{lem}
\begin{proof}
We will calculate the set $q(\sigma)$ for various $k$-faces $\sigma$. Let
\[
\sigma = \langle  b_{i_0},  b_{i_1},\dots, b_{i_{k}}\rangle.
\]
Then we have:
\[ \sigma = \left\{ \sum_{\ell=0}^k s_{i_\ell}  b_{i_\ell} \mid s_{i_\ell} \in [0,1],\;\sum_{\ell=0}^k s_{i_\ell} = 1 \right\}. \]
Because $ b_j =  e_1 + \dots +  e_j$, we can rearrage the sums above into standard coordinates to obtain:
\[
\sigma = \left\{ \sum_{j=1}^n \left( \sum_{i_m \ge j} s_{i_m} \right)  e_j \mid s_{i_\ell} \in [0,1],\;\sum_{\ell=0}^k s_{i_\ell} = 1 \right\}.
\]
Applying $q$ to the above turns these vectors into the sets of their coordinates:
\begin{equation}\label{qsigma}
q(\sigma)= \left\{ \left\{ \sum_{i_m \ge j} s_{i_m}\text{ mod }1 \mid j \in \{1,\dots, n\} \right\} \mid s_{i_\ell} \in [0,1],\;\sum_{\ell=0}^k s_{i_\ell} = 1 \right\}.
\end{equation}

When $\sigma$ is extremal, we have $i_0=0$ and $i_k =n$. In this case, the set appearing inside \eqref{qsigma} takes the form:
\[  \left\{ \sum_{i_m \ge j} s_{i_m} \text{ mod }1 \mid j \in \{1,\dots, n\} \right\} = \left\{\sum_{j=1}^k s_{i_j}, \sum_{j=2}^k s_{i_j}, \dots, s_{i_k} \right\} \text{mod } 1, \]
which is a set of $k$ elements. (Note that the parameter $s_{i_0}$ does not appear at all, since it is the coefficient on $ b_0 =  0$.) Renaming these parameters, we have $q(\sigma) = E$, where $E$ is the set:
\[ E = \left\{ \{ t_{1} + \dots + t_{k}, t_{2}+\dots + t_{k}, \dots, t_{k} \} \text{ mod }1 \mid t_\ell \in [0,1],\;\sum_{\ell=1}^k t_\ell = 1 \right\}. \]
We see that $E$ is independent of the choices of the indices $i_1,\dots,i_{k-1}$, and thus we will have $q(\sigma)=E$ for any extremal $k$-face $\sigma$.

Let $E'$ be the following set:
\[ E' = \left\{ \{ t_{1} + \dots + t_{k}, t_{2}+\dots + t_{k}, \dots, t_{k}, 0 \} \text{ mod }1 \mid t_\ell \in [0,1],\;\sum_{\ell=1}^k t_\ell = 1 \right\}. \]
We complete the proof by showing that any nonextremal $k$-face $\sigma$ has $q(\sigma)=E'$.

First assume that $\sigma$ is non-extremal because its vertices do not include $b_n$, but do include $b_0$. In this case again the $s_{i_0}$ parameter does not appear, and the set appearing inside \eqref{qsigma} takes the form:
\[  \left\{ \sum_{i_m \ge j} s_{i_m} \text{ mod }1 \mid j \in \{1,\dots, n\} \right\} = \left\{\sum_{j=1}^k s_{i_j}, \sum_{j=2}^k s_{i_j}, \dots, s_{i_k},0 \right\} \text{mod }1, \]
where the 0 appears because the sum $\sum_{i_m\ge n} s_{i_m}$ has no terms because $ b_n$ is not a vertex of $\sigma$. Renaming parameters as above, we see that $q(\sigma) = E'$ as desired.

Next we consider a $k$-face $\sigma$ which is non-extremal because its vertices do not include $b_0$, but do include $b_n$. In this case since $i_0 \ge 1$, the set appearing inside \eqref{qsigma} takes the form:
\[  \left\{ \sum_{i_m \ge j} s_{i_m} \text{ mod }1 \mid j \in \{1,\dots, n\} \right\} = \left\{\sum_{j=1}^k s_{i_j}, \sum_{j=2}^k s_{i_j}, \dots, s_{i_k} \right\}  \text{mod }1. \]
Since $s_{i_0} + \dots + s_{i_k} = 1$, the above becomes:
\[ \{ 1, s_{i_1}+\dots + s_{i_k}, \dots, s_{i_k} \} \text{ mod }1 = \{ s_{i_1}+\dots + s_{i_k}, \dots, s_{i_k}, 0 \} \text{ mod }1 \]
and so $q(\sigma) = E'$ in this case as well.

Finally we consider a $k$-face $\sigma$ which is non-extremal because its vertices do not include $b_0$ or $b_n$. In this case, for simiar reasons as above, the set appearing inside \eqref{qsigma} takes the form:
\[  \left\{ \sum_{i_m \ge j} s_{i_m} \text{ mod }1 \mid j \in \{1,\dots, n\} \right\} = \left\{1, \sum_{j=1}^k s_{i_j}, \sum_{j=2}^ks_{i_j}, \dots, s_{i_k}, 0 \right\} \text{ mod }1. \]
and since $1=0 \mod 1$, we have $q(\sigma)=E'$ as desired.

\end{proof}

The lemma means that in each dimension $k$ with $0<k<n$, we have exactly two $k$-simplices in $C_n(S^1)$: one represented by an extremal $k$-face, and one by a nonextremal $k$-face. In each dimension $k$, we will choose specific representatives:
\[
\begin{split}
\sigma_k &= q(\langle b_0, b_1, \dots, b_k \rangle) \\
\rho_k &= q(\langle b_0, b_1, \dots, b_{k-1}, b_n \rangle)
\end{split}
\]

Next we discuss the fundamental group of $C_n(S^1)$. For $n=1$ we may identify $C_1(S^1)$ with $S^1$, so we have $\pi_1(C_1(S^1)) \cong \Z$. For $n=2$, by Proposition \ref{samen=2} we may identify $C_2(S^1)=\SP^2(S^1)$. This space is known to be homotopy equivalent to $S^1$ (see \cite{blagrz2004}), and so we have $\pi_1(C_2(S^1)) \cong \Z$. For $C_n(S^1)$ with $n\ge 3$, we require new arguments.

\begin{lem}\label{edgeslemma}
For $n\ge 3$, there are two edges in the complex $C_n(S^1)$ given by $\sigma_1$ and $\rho_1$. The nonextremal edge can be parameterized by $\alpha(t)=\{0,t\}$ for $t\in [0,1]$, and the extremal edge can be parameterized by $\beta(t) = \{t\}$ for $t\in [0,1]$.
\end{lem}
\begin{proof}
Lemma \ref{faceslemma} shows that $\sigma_1$ and $\rho_1$ are the only edges in $C_n(S^1)$. We need only demonstrate that they have the parameterizations described above.
%
%

Consider the non-extremal edge $\sigma = \langle b_0,b_1\rangle$ of $\Delta_n$. This is the edge from $b_0 = (0,\dots,0)$ to $b_1=(1,0,\dots,0)$.
Then the standard parameterization of $\sigma$ is $a:[0,1] \to \Delta_n$ given by $a(t) = (t,0,\dots,0)$,
and so $\sigma_1 = q(\sigma)$ is parameterized by $\alpha = q\circ a$ given by:
\[ \alpha(t) = \{0,t\}. \]

For the extremal edge $\rho_1 = q(\langle b_0, b_n\rangle)$, the parameterization of $\langle b_0,b_n\rangle$ is given by $b(t) = (t, \dots, t)$,
and so $\rho_1$ is parameterized by $\beta(t) = \{t\}$ as desired.
\end{proof}

\begin{thm}\label{simplyconnected}
For $n\ge 3$, the fundamental group  $\pi_1(C_n(S^1))$ is trivial.
\end{thm}
\begin{proof}
We need only focus our attention on the edges and faces of $C_n(S^1)$. By Lemma \ref{faceslemma} there are two faces $\sigma_2,\rho_2$, and these can be represented in $\Delta_n$ by $A = \langle b_0,b_1,b_2\rangle$ and $B=\langle b_0,b_1,b_n\rangle$. These are pictured in Figure \ref{facesfig}, in which $a, b$ are as in Lemma \ref{edgeslemma} and $a^{-1}$ and $b^{-1}$ denote the reverse of the paths $\alpha$ and $\beta$.
\begin{figure}

\[
\begin{tikzpicture}

\foreach \theta in {-30, 90, 210} {
 \fill (\theta:2) circle (.05cm);
 }

\node[below left] at (210:2) {$b_0$};
\node[above] at (90:2) {$b_1$};
\node[below right] at (-30:2) {$b_2$};

\fill[fill=lightgray] (-30:2) -- (90:2) -- (210:2) -- cycle;

\begin{scope}[thick,decoration={
    markings,
    mark=at position 0.5 with {\arrow{>}}}
    ]
\draw[postaction={decorate}] (210:2) -- (90:2) node[pos=.5,above left] {$a$};
\draw[postaction={decorate}] (90:2) -- (-30:2) node[pos=.5,above right] {$a$};
\draw[postaction={decorate}] (-30:2) -- (210:2) node[pos=.5,below] {$a^{-1}$};
\end{scope}

\node at (0,0) {$A$};
\end{tikzpicture}
\qquad
\begin{tikzpicture}

\foreach \theta in {-30, 90, 210} {
 \fill (\theta:2) circle (.05cm);
 }

\node[below left] at (210:2) {$b_0$};
\node[above] at (90:2) {$b_1$};
\node[below right] at (-30:2) {$b_n$};

\fill[fill=lightgray] (-30:2) -- (90:2) -- (210:2) -- cycle;

\begin{scope}[thick,decoration={
    markings,
    mark=at position 0.5 with {\arrow{>}}}
    ]
\draw[postaction={decorate}] (210:2) -- (90:2) node[pos=.5,above left] {$a$};
\draw[postaction={decorate}] (90:2) -- (-30:2) node[pos=.5,above right] {$a$};
\draw[postaction={decorate}] (-30:2) -- (210:2) node[pos=.5,below] {$a^{-1}$};
\end{scope}

\node at (0,0) {$B$};
\end{tikzpicture}
\]
\caption{Representatives for the two faces of $C_n(S^1)$.\label{facesfig}}
\end{figure}
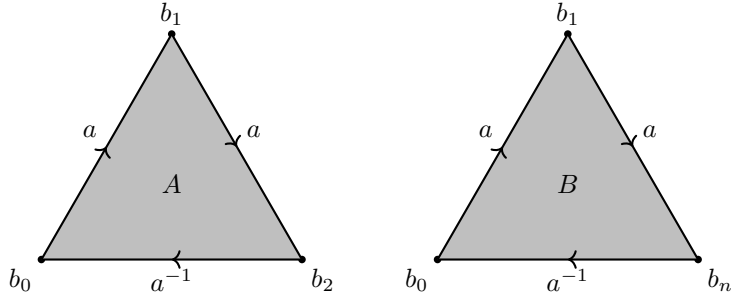

The loop traversing the boundary of $\sigma_1$ is given by $\alpha * \alpha * \alpha^{-1} \simeq \alpha$, and the loop traversing the boundary of $\sigma_2$ is given by $\alpha * \alpha * \beta^{-1}$. Thus by Van Kampen's theorem, we have the following presentation for the fundamental group:
\[ \pi_1(C_n(S^1)) = \langle \alpha, \beta \mid \alpha, \alpha^2\beta^{-1} \rangle. \]
From the second relation we can eliminate $\beta$ and we obtain  $\langle \alpha \mid \alpha \rangle$, and the result follows.
\end{proof}

For a complex $X$, let $S_k(X)$ denote the simplicial chain group of $X$ in dimension $k$.

Lemma \ref{faceslemma} means that for $0<k<n$, the group $S_k(C_n(S^1))$ has two generators $\sigma_k$ and $\rho_k$. In the top dimension $k=n$ we have a single $n$-simplex which we denote $\rho_n = q(\Delta_n)$. The next lemma computes the boundary maps:
\begin{lem}\label{boundarymaps}
With notations as above, for $0<k<n$, we have:
\[
\partial \sigma_k =  \begin{cases} \sigma_{k-1} &$ if $k$ is even,$ \\
0 &$ if $k$ is odd.$ \end{cases}
\qquad
\partial \rho_k =  \begin{cases} 2\sigma_{k-1} - \rho_{k-1} &$ if $k$ is even,$ \\
0 &$ if $k$ is odd.$ \end{cases} \]
And we have
\[ \partial \rho_n = \begin{cases} 2\sigma_{n-1} - \rho_{n-1} &$ if $n$ is even,$ \\
0 &$ if $n$ is odd.$ \end{cases} \]
\end{lem}
\begin{proof}
First we prove the statement for $\sigma_k$ and $\rho_k$ with $0<k<n$. We compute:
\[ \partial \sigma_k = \sum_{i=0}^k (-1)^i q(\langle b_0, \dots, \hat b_i, \dots, b_k \rangle) \]
where $\hat b_i$ denotes omission of this coordinate. The $(k-1)$-faces appearing inside the sum are all non-extremal because they all lack the vertex $b_n$. Thus we have
\[ \partial \sigma_k = \sum_{i=0}^k (-1)^i \sigma_{k-1}
= \begin{cases} \sigma_{k-1} &$ if $k$ is even,$ \\
0 &$ if $k$ is odd$ \end{cases}
\]
as desired.

Similarly we compute $\partial \rho_k$ as an alternating sum of various omissions of $\rho_k = q(\langle b_0, b_1, \dots, b_{k-1}, b_n \rangle)$. The first and last of these will be nonextremal, and the others will be extremal. We obtain:
\[ \partial \rho_k = \sigma_k + (-1)^k \sigma_k + \sum_{i=1}^{k-1} \rho_k = \begin{cases} 2\sigma_{k-1} - \rho_{k-1} &$ if $k$ is even,$ \\
0 &$ if $k$ is odd,$ \end{cases}
\]
as desired.

For the statement concerning $\rho_n$ we compute as above starting with $\rho_n = q(\langle b_0, \dots, b_n \rangle)$. Since this simplex includes both $b_0$ and $b_n$, the calculation proceeds exactly as in the case above for an extremal face. The desired formula follows.
\end{proof}

The calculations above allow us to compute the homology groups easily. The cases $n=1$ and $n=2$ are straightforward because $C_1(S^1)$ can be identified with $S^1$, and as discussed above $C_2(S^1)$ can be identified with $\SP^2(S^1)$ which is homotopy equivalent to $S^1$. Thus for $n \in \{1,2\}$ we have:
\[ \begin{split}
H_0(C_n(S^1)) &= H_1(C_n(S^1)) = \Z, \\
H_k(C_n(S^1)) &= 0 \text{ for $k\not\in \{0,1\}$.}
\end{split} \]

For $n\ge 3$ our result is as follows:
\begin{thm}\label{mainHom}
The homology groups of $C_n(S^1)$ are as follows. When $n$ is even, we have:
\[
\begin{split}
H_0(C_n(S^1)) &= H_{n-1}(C_n(S^1)) = \Z, \\
H_k(C_n(S^1)) &= 0 \text{ for $k\not\in \{0,n-1\}$}.
\end{split}
\]
When $n$ is odd, we have:
\[
\begin{split}
H_0(C_n(S^1)) &= H_n(C_n(S^1)) = \Z, \\
H_k(C_n(S^1)) &= 0 \text{ for $k\not\in \{0,n\}$.}
\end{split}
\]
That is, $C_n(S^1)$ is a homology $(n-1)$-sphere when $n$ is even, and a homology $n$-sphere when $n$ is odd.
\end{thm}
\begin{proof}
Lemma \ref{boundarymaps} shows that, for $0<k < n$, the boundary map
\[
\partial_k:S_k(C_n(S^1)) \to S_{k-1}(C_n(S^1))
\]
is an isomorphism when $k$ is even, and 0 when $k$ is odd. This means that $H_0(C_n(S^1)) = \Z$, and $H_k(C_n(S^1)) = 0$ for $k<n-1$.

It remains to consider dimensions $n$ and $n-1$, in which the chain complex is:
\[
0 \xrightarrow{\partial_{n+1}} S_n(C_n(S^1)) \xrightarrow{\partial_n} S_{n-1}(C_n(S^1)) \xrightarrow{\partial_{n-1}} \dots
\]
where $S_n(C_n(S^1)) \cong \Z$ and $S_{n-1}(C_n(S^1)) \cong \Z^2$.

If $n$ is even, then $\partial_n$ is injective, so $H_n(C_n(S^1)) = 0$. The calculations of Lemma \ref{boundarymaps} show that $\im(\partial_n) \cong \Z$,  
$\ker(\partial_{n-1}) = S_{n-1}(C_n(S^1)) = \Z^2$ because $\partial_{n-1}=0$, and $\im(\partial_n)$   is a primitive subgroups of   $\Z^2$. So follows  that $H_{n-1}(C_n(S^1)) = \Z$.

If $n$ is odd then $\partial_n$ is zero while  $\partial_{n-1}$ is isomorphism. This means that $H_n(C_n(S^1)) = S_n(C_n(S^1)) = \Z$, and $H_{n-1}(C_n(S^1)) = 0$.


\end{proof}

Using the Hurewicz and Whitehead theorems, the result above will show that $C_n(S^1)$ in fact has the homotopy type of the appropriate spheres.

\begin{cor} The  space   $C_n(S^1)$ has the homotopy type of $S^1$ for $n=1,2$. For  $n\geq 3$ it  has
 the homotopy type of  $n$-sphere when $n$ is odd, and the  homotopy type of     $(n-1)$-sphere when $n$ is even.
\end{cor}
\begin{proof}
Let $k=n-1$ if $n$ is even, and $k=n$ if $n$ is odd. We will show that $C_n(S^1)$ has the homotopy type of $S^k$.
The statement for $n=1$ is clear because $C_1(S^1)=S^1$. For $n=2$, as mentioned above, by Proposition \ref{samen=2} we may identify $C_2(S^1)$ with $\SP^2(S^1)$, which is known to be homotopy equivalent to $S^1$ (see \cite{blagrz2004}). So it remains to consider the case $n\ge 3$.

Theorem \ref{simplyconnected} has already shown that $C_n(S^1)$ is simply connected when $n\ge 3$. By the Hurewicz theorem, the first nonzero homotopy and homology groups must occur in the same dimension, and they are isomorphic.
By Theorem \ref{mainHom} this means that $\pi_i(C_n(S^1)) = 1$ for $i<k$ and $\pi_k(C_n(S^1)) = \Z = \pi_k(S^k)$.

Let $f:S^k \to C_n(S^1)$ be a map which realizes the isomorphism in the $k$th homotopy group. In fact this $f$ induces an isomorphism in homology in all dimensions, because we have already shown that the homology in all other dimensions is zero.  Then by the Whitehead theorem, $f$ gives a homotopy equivalence of $S^k$ and $C_n(S^1)$ as desired.
\end{proof}

\end{document}